\documentclass[preprint, 1p]{elsarticle}

\usepackage{epsf}
\usepackage{graphicx}
\usepackage{dcolumn}
\usepackage{bm}
\usepackage{subfigure}
\usepackage{color}
\usepackage{amsthm}
\usepackage{amssymb}
\usepackage{amsmath}
\usepackage{pb-diagram}
\usepackage{algorithm}
\usepackage{algorithmic}
\usepackage{hyperref}
\usepackage{bbm}
\usepackage[mathscr]{euscript}
\usepackage{enumerate}
\usepackage{tikz-cd}
\usepackage{array}

\usepackage[normalem]{ulem}

\theoremstyle{plain} 
\numberwithin{equation}{section}
\newtheorem{thm}[equation]{Theorem}
\newtheorem{cor}[equation]{Corollary}
\newtheorem{lem}[equation]{Lemma}
\newtheorem{prop}[equation]{Proposition}

\theoremstyle{definition}
\newtheorem{defn}[equation]{Definition}

\theoremstyle{remark}
\newtheorem{rem}[equation]{Remark}

\newcommand{\E}{{\mathbb{E}}}

\newcommand{\N}{{\mathbb{N}}}

\newcommand{\R}{{\mathbb{R}}}

\newcommand{\Z}{{\mathbb{Z}}}

\newcommand{\cB}{{\mathcal B}}

\newcommand{\cD}{\mathcal{D}}

\newcommand{\cK}{{\mathcal K}}

\newcommand{\cR}{{\mathcal R}}

\newcommand{\cT}{{\mathcal T}}

\newcommand{\cX}{{\mathcal X}}
\newcommand{\cY}{{\mathcal Y}}

\newcommand{\sC}{{\mathsf C}}

\newcommand{\sk}{{\mathsf k}}

\newcommand{\pd}{{\mathsf{PD}}}
\newcommand{\pdbf}{{\mathbf{PD}}}

\newcommand{\im}{\mathrm{im\;}}
\newcommand{\coker}{\mathrm{coker\;}}

\newcommand{\dom}{\mathrm{dom\;}}

\newcommand{\X}{\mathcal{X}}

\newcommand{\setof}[1]{\left\{ {#1} \right\}}
\newcommand{\interval}[2]{\langle {#1}, {#2}\rangle}

\newcommand{\conv}[1]{\text{Conv}({#1})}

\newcommand{\overlap}{\lhd\kern-4.5pt{\blacktriangleleft}\text{\;}}
\newcommand{\bb}{\lhd}
\newcommand{\ba}{\blacktriangleleft}

\newcommand{\bra}[1]{\left<{#1}\right\vert}
\newcommand{\ket}[1]{\left\vert{#1}\right>}

\makeatletter
\newcommand*{\matching}{%
  \mathrel{%
    \mathpalette\@vneq{\rightarrow}%
  }%
}
\newcommand*{\@vneq}[2]{%
  \sbox0{\raisebox{\depth}{$#1\neq$}}%
  \sbox2{\raisebox{\depth}{$#1|\,\m@th$}}%
  \ifdim\ht2>\ht0 %
    \sbox2{\resizebox{\vneqxscale\width}{\vneqyscale\ht0}{\unhbox2}}%
  \fi
  \sbox2{$\m@th#1\vcenter{\copy2}$}%
  \ooalign{%
    \hfil\phantom{\copy2}\hfil\cr
    \hfil$#1#2\m@th$\hfil\cr
    \hfil\copy2\hfil\cr
  }%
}
\newcommand*{\vneqxscale}{1}
\newcommand*{\vneqyscale}{.3}

\def\ps@pprintTitle{%
 \let\@oddhead\@empty
 \let\@evenhead\@empty
 \def\@oddfoot{\centerline{\thepage}}%
 \let\@evenfoot\@oddfoot}

\makeatother

\newcolumntype{L}{>{\centering\arraybackslash}m{4cm}}

\usetikzlibrary{matrix,arrows,positioning,calc,decorations.markings}
\tikzset{strike thru/.style={
    decoration={markings, mark=at position 0.5 with {
        \draw [-]
            ++ ( 0pt,-1.25pt)
            -- ( 0pt, 1.25pt);}
    },
    postaction={decorate},
}}

\begin{document}

\begin{frontmatter}
\title{A Comparison Framework for Interleaved Persistence Modules}

\author[ru]{Shaun Harker}
\ead{sharker@math.rutgers.edu}

\author[tu]{Miroslav Kram\'{a}r}
\ead{miroslav.kramar@inria.fr}

\author[up]{Rachel Levanger}
\ead{levanger@seas.upenn.edu}

\author[ru]{Konstantin Mischaikow}
\ead{mischaik@math.rutgers.edu}

\address[ru]{Department of Mathematics,
Hill Center-Busch Campus,
Rutgers University,
110 Frelingheusen Rd,
Piscataway, NJ  08854-8019, USA}

\address[tu]{INRIA Saclay,
1 Rue Honoré d'Estienne d'Orves, 91120 Palaiseau, France}

\address[up]{Department of Electrical and Systems Engineering,
University of Pennsylvania,
200 S. 33rd St.,
Philadelphia, PA 19104-6314, USA}

\begin{abstract}

We present a generalization of the induced matching theorem of \cite{bauer_lesnick} and use it to prove a generalization of the algebraic stability theorem for $\R$-indexed pointwise finite-dimensional persistence modules. 
Via numerous examples, we show how the generalized algebraic stability theorem enables the computation of rigorous error bounds in the space of persistence diagrams that go beyond the typical formulation in terms of bottleneck (or log bottleneck) distance. 
\end{abstract}

\end{frontmatter}

\section{Introduction}
\label{sec:intro}
Persistent homology \cite{edelsbrunner, carlsson_zomorodian, oudot} is a key element in the rapidly-developing field of topological data analysis, where it is used both as a means of identifying geometric structures associated with data and as a data reduction tool.
Any work with data involves approximations that arise from finite sampling, limits to measurement, and experimental or numerical errors.
The results of this paper focus on obtaining rigorous bounds on the variations in persistence diagrams arising from these approximations.

To motivate this work, we begin with the observation that many problems in data analysis can be rephrased as a problem concerned with the analysis of the geometry induced by a scalar function $f\colon X\to \R$ defined on a topological space $X$.
Two canonical examples are as follows.
Assume that $(X,\rho)$ is a metric space and let $\cX\subseteq X$. 
Single-linkage hierarchical clustering problems based on $\cX$ are naturally associated with the function $f\colon X \to [0,\infty)$ given by
\[
f(x) := \rho(x, \cX) = \inf\setof{\rho(x,\xi) : \xi \in \cX},
\]
where clusters are derived from the connected components of the sublevel set
\[
\sC(f, t ) := \setof{x\in X : f(x) \leq t}
\]
for choices of $t\in [0,\infty)$. 
The collection $\{ \sC(f, t )\}_{t \in \R}$ is called the \emph{sublevel set filtration} of $X$ induced by $f$. 
Superlevel sets and superlevel set filtrations are defined similarly by considering the sets $\setof{x\in X : t \leq f(x)}$ for every $t \in \R$.

Alternatively, assume that $X$ is a  topological domain and $f\colon X\to \R$ is a scalar value of a nonlinear physical model,
e.g.\ the magnitude of vorticity or temperature field of a fluid, the chemical density in a reaction diffusion system, the magnitude of forces between particles in a granular system, etc.
Patterns produced by these systems are often associated with sublevel or superlevel sets of $f$.
In fact, the direct motivation for this work is to justify claims made in \cite{physicaD} concerning the time-evolution of patterns in convection models.

These examples are meant to motivate our interest in studying the geometry of the sets $\sC(f,  t)$.
Homology provides a coarse but computable representation of this geometry.
In particular, for each $t\in\R$, there is an assigned graded vector space
\[
M(f)_t = H_\bullet (\sC(f,  t),\sk),
\]
where $\sk$ is a field.
Because each $t\leq s$ implies $\sC(f,  t) \subseteq \sC(f,  s)$, the inclusion maps induce the following linear maps at the level of homology:
\[
\varphi_{M(f)}(t,s)\colon M(f)_t \to M(f)_s.
\]
This homological information can be abstracted as follows.

\begin{defn}
\label{defn:persistence_module}
A \emph{persistence module} is a collection of vector spaces indexed by the real numbers, $\{V_t\}_{t \in \R}$, and linear maps $\setof{\varphi_V(s, t) : V_s \rightarrow V_t}_{s \leq t \in \R}$ satisfying the following conditions:
\begin{enumerate}[(i)]
\item $\varphi_V(t, t) = \text{id}_{V_t}$ for every $t \in \R$, and 
\item $\varphi_V(s, t) \circ \varphi_V(r, s) = \varphi_V(r, t)$ for every $r \leq s \leq t$ in $\R$. 
\end{enumerate}
We write $(V, \varphi_V)$ to denote the collection of vector spaces and compatible linear maps, and will sometimes just write $V$ for the full persistence module when the maps are clear.  
We say that $V$ is a \emph{pointwise finite dimensional (PFD)} persistence module when every $V_t$ is finite-dimensional.
\end{defn}

As is described  in  Sections~\ref{subsec:2-preliminaries-Morphisms-Diagrams} and \ref{sec:algebraic_stability}, a PFD persistence module gives rise to a \emph{persistence diagram}, which is a set of points in $\overline{\R}^2 \times \N$, where $\overline{\R} = \R\cup\setof{-\infty,\infty}$. 
Given a PFD persistence module $(V, \varphi_V)$, we denote its associated persistence diagram by $\pd(V)$. 

Observe that we have outlined a procedure by which the sublevel sets of a scalar field $f$ produce a persistence diagram $\pd$.
Returning to our examples, in the first case, it is reasonable to assume that the actual available data is $\cX' \subseteq \cX\subseteq X$, as opposed to  $\cX$, which represents the true set of objects upon which the clustering is to be based. 
In this case, collecting experimental or numerical data results in $f'\colon X\to \R$, an approximation of the actual function of interest, $f$.
Recent computational developments have led to the routine computation of $\pd'$, the persistence diagram associated with $\cX'$ or $f'$. 
Thus, the natural question is this: how is $\pd'$, the computed persistence diagram, related to  $\pd$, the persistence diagram of interest?

A fundamental result \cite{edelsbrunner} in the theory of persistent homology is that a variety of metrics can be imposed on the space of persistence diagrams such that  $\pd$ changes continuously with respect to $L^\infty$ changes in $f$.
However, these metrics provide limited control on the variation of individual points in the persistence diagram.
Recent developments by Bauer and Lesnick \cite{bauer_lesnick} allow for comparisons of persistence modules through a matching of the associated persistence points.
The primary theoretical results of this paper, Theorem~\ref{thm:geninducedmatching}  and Theorem~\ref{thm:genalgstability}, are extensions of Bauer and Lesnick's Induced Matching Theorem and Algebraic Stability Theorem, respectively.

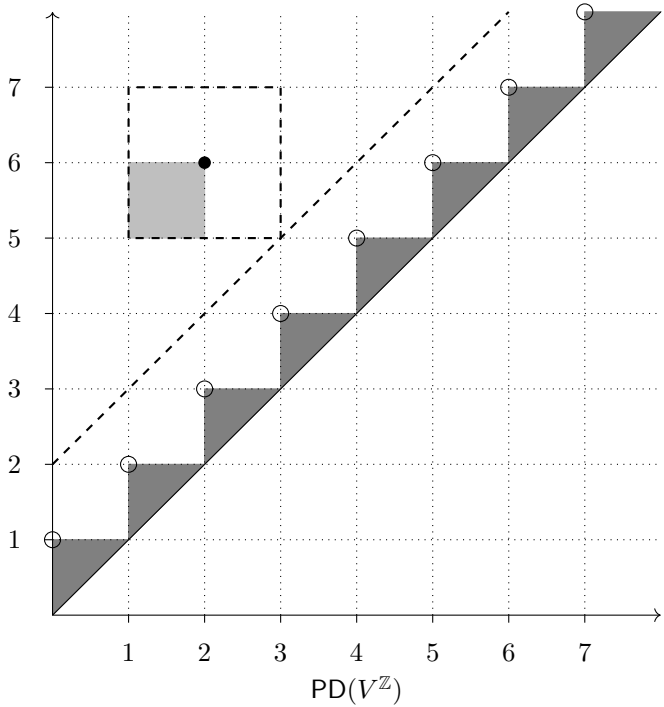
\begin{figure}[t]
\begin{picture}(300,250)

\put(50,0){
\begin{tikzpicture}
[scale=1]

		\fill[draw=black,color=gray]  (0,0) -- (0,1) -- (1,1) -- (0,0);
		\fill[draw=black,color=gray]  (1,1) -- (1,2) -- (2,2) -- (1,1);
		\fill[draw=black,color=gray]  (2,2) -- (2,3) -- (3,3) -- (2,2);
		\fill[draw=black,color=gray]  (3,3) -- (3,4) -- (4,4) -- (3,3);
		\fill[draw=black,color=gray]  (4,4) -- (4,5) -- (5,5) -- (4,4);
		\fill[draw=black,color=gray]  (5,5) -- (5,6) -- (6,6) -- (5,5);
		\fill[draw=black,color=gray]  (6,6) -- (6,7) -- (7,7) -- (6,6);
		\fill[draw=black,color=gray]  (7,7) -- (7,8) -- (8,8) -- (7,7);
		
		\fill[draw=black,color=lightgray]  (1,5) -- (1,6) -- (2,6) -- (2,5) -- (1,5);

		\draw [<->]  (8,0) --(0,0)  --(0,8);

		
		\draw [] (0,0) -- (8,8);
		
		 
		 \draw [dotted] (1,0) --(1,8); 
		 \draw [dotted] (2,0) --(2,8); 
		 \draw [dotted] (3,0) --(3,8); 
		 \draw [dotted] (4,0) --(4,8); 
		 \draw [dotted] (5,0) --(5,8); 
		 \draw [dotted] (6,0) --(6,8); 
		 \draw [dotted] (7,0) --(7,8); 
		 
		 \draw [] (1,0) --(1,-0.1); 
		 \draw [] (2,0) --(2,-0.1); 
		 \draw [] (3,0) --(3,-0.1); 
		 \draw [] (4,0) --(4,-0.1); 
		 \draw [] (5,0) --(5,-0.1); 
		 \draw [] (6,0) --(6,-0.1); 
		 \draw [] (7,0) --(7,-0.1); 
		 
		 \draw(1,-0.5) node{$1$}; 
		 \draw(2,-0.5) node{$2$}; 
		 \draw(3,-0.5) node{$3$}; 
		 \draw(4,-0.5) node{$4$}; 
		 \draw(5,-0.5) node{$5$}; 
		 \draw(6,-0.5) node{$6$}; 
		 \draw(7,-0.5) node{$7$}; 
		 
		 \draw [dotted] (0,1) --(8,1); 
		 \draw [dotted] (0,2) --(8,2); 
		 \draw [dotted] (0,3) --(8,3); 
		 \draw [dotted] (0,4) --(8,4); 
		 \draw [dotted] (0,5) --(8,5); 
		 \draw [dotted] (0,6) --(8,6); 
		 \draw [dotted] (0,7) --(8,7); 
		  
		 \draw [] (0,1) --(-0.1,1); 
		 \draw [] (0,2) --(-0.1,2); 
		 \draw [] (0,3) --(-0.1,3); 
		 \draw [] (0,4) --(-0.1,4); 
		 \draw [] (0,5) --(-0.1,5); 
		 \draw [] (0,6) --(-0.1,6); 
		 \draw [] (0,7) --(-0.1,7); 
	 
		 \draw(-0.5, 1) node{$1$};
		 \draw(-0.5, 2) node{$2$};
		 \draw(-0.5, 3) node{$3$};
		 \draw(-0.5, 4) node{$4$};
		 \draw(-0.5, 5) node{$5$};
		 \draw(-0.5, 6) node{$6$};
		 \draw(-0.5, 7) node{$7$};	
		 
		 \node at (2,6) [circle,fill, scale=0.5] {};
		 \draw (0,1) circle (3pt);
		 \draw (1,2) circle (3pt);
		 \draw (2,3) circle (3pt);
		 \draw (3,4) circle (3pt);
		 \draw (4,5) circle (3pt);
		 \draw (5,6) circle (3pt);
		 \draw (6,7) circle (3pt);
		 \draw (7,8) circle (3pt);
		 
		  \draw [thick,dashed] (0,2) --(6,8);
		  \draw [thick,dashed] (1,5) --(1,7) -- (3,7) -- (3,5) -- (1,5);

		\draw(4,-1) node{$\pd(V^\Z)$};   
\end{tikzpicture}
}		
		
\end{picture}
\caption{ A schematic diagram illustrating the potential locations of persistence points of $\pd(V)$ based on the computation of $\pd(V^{\Z})$. 
The persistence point $(2,6)$ of $\pd(V^\Z)$ is matched with a persistence point of $\pd(V)$ which must lie in the light gray region.
If a persistence point of $\pd(V^\Z)$ occurs at one of the open circles, then it is possible that it is a computational artifact, i.e.\ that it is not matched with any persistence point in $\pd(V)$.
The dark gray region indicates the potential location of persistence points of $\pd(V)$ that cannot be detected because of the integer valued approximation used to compute $\pd(V^\Z)$.
Finally,  if $W$ is an arbitrary PFD persistence module and the bottleneck distance between $\pd(V^\Z)$ and $\pd(W)$ is one, then  $\pd(W)$ may have a single point in the region indicated by the dashed square, and arbitrarily many persistence points in the region below the dashed line.
}
\label{fig:Zpd}
\end{figure}

As indicated above, the applications of these extensions provided the motivation for this paper. 
To give a particular example, consider the persistence module $V=(M(f),\varphi_{M(f)})$ associated with the scalar function $f\colon X\to \R$.
However, assume that we are only able to sample the sublevel sets of $f$ at the integers $\Z\subset \R$.
As explained in Section~\ref{sec:example_discretize}, this sampling gives rise to a persistence module $V^\Z$.
Figure~\ref{fig:Zpd} indicates the type of result that we obtain.
What is shown is the persistence diagram $\pd(V^\Z)$ (which can be computed) which, for the region shown, is assumed to have a single persistence point at $(2,6)$.
As a consequence of Proposition~\ref{prop:discretized_pm}, we can conclude that the persistence diagram of interest, $\pd(V)$, contains a single persistence point in the light gray region.
It is also possible that $\pd(V)$ contains persistence points in the dark gray regions.  
This would correspond to geometrical features of $f\colon X\to \R$ that take place on a scale that is to fine to be detected by the integer-valued sampling.
Finally, if a persistence point for $\pd(V^\Z)$ occurred at one of the open circles centered at $(n,n+1)$, then this persistence point could be a computational artifact, i.e.\ it is not necessarily associated with any persistence point of $\pd(V)$.

Figure~\ref{fig:Zpd} also indicates the advantage of comparing persistence diagrams using the matching theorems of this paper as opposed to the classical metrics such as the bottleneck distance \cite{Cohen-Steiner}.
In particular, if $\pd(W)$ is an arbitrary persistence diagram whose bottleneck distance from $\pd(V^\Z)$ is one, then  $\pd(W)$ may have a single point in the region indicated by the dashed square and arbitrarily many persistence points in the region below the dashed line, versus a single point in the light gray box and arbitrarily many points in the dark gray region.

An outline of this paper is as follows.
In Section~\ref{sec:preliminaries} we review the essential concepts associated with persistence modules required for our results.
This section defines the notions of persistence modules and their morphisms, interleavings, and induced matchings.
Of particular note is the introduction of the concept of a non-constant  translation pair that is used to extend the results of Bauer and Lesnick \cite{bauer_lesnick}, where translation pairs are defined in terms of uniform translations.
We also include a review of Galois connections, as we use these concepts for some proofs in Section~\ref{sec:algebraic_stability}.

Section~\ref{sec:gen_induced_matching} focuses on Theorem~\ref{thm:geninducedmatching}, which is an extension of the Induced Matching Theorem of  \cite{bauer_lesnick}.
The proof incorporates ideas from the theory of generalized interleavings of Bubenik, de Silva, and Scott \cite{bubenik_desilva_scott}. 

Section~\ref{sec:algebraic_stability} begins with the proof of Theorem~\ref{thm:genalgstability}, which follows closely the proof of the Algebraic Stability Theorem of  \cite{bauer_lesnick}.
The remainder of the section provides results, corollaries, and re-interpretations of Theorem~\ref{thm:genalgstability}.
In particular, under the assumption that the maps in the translation pair are invertible, Corollary~\ref{cor:genalgstab_invertible} provides an easy-to-state version of Theorem~\ref{thm:genalgstability} that clarifies how translation pairs relate to stability in the space of persistence diagrams.
Proposition~\ref{prop:Matching} and Corollary~\ref{cor:genalgstab_invertible_pi} indicate how Theorem~\ref{thm:genalgstability} applies to specific points in the associated persistence diagrams.

Finally, Section~\ref{sec:examples} provides examples of applications of Theorem~\ref{thm:genalgstability}.
As indicated above Section~\ref{sec:example_discretize} considers the problem of bounds on the desired persistence diagram under the assumption that  values of the function $f\colon X\to \R$ can only be sampled discretely.

In Section~\ref{sec:example_point_clouds}, we consider the following problem associated with the first example of this introduction.
Assume that one is given a large finite point cloud $\cX\subset X$ for which one wishes to compute the persistence diagram $\pd(V)$ for the persistence module $V= (M(f),\varphi_{M(f)})$.
However, because of the size of $\cX$, the computational cost of computing $\pd(V)$ is prohibitive.
At the time of this writing, this is a reasonable concern since the standard approach is to use a Vietoris-Rips complex (this is discussed at the beginning of Section~\ref{sec:example_point_clouds}) to compute $\pd(V)$, and the size of this complex grows extremely fast as a function of the size of $\cX$ and the magnitude of $f$.
This suggests that once the magnitude of $f$ is too large, then one should subsample  and compute an approximate persistence diagram $\pd(V')$ based on $\cX'\subset \cX$.
Proposition~\ref{prop:subsample_interleaving} provides a simple result bounding the locations of the persistence points in $\pd(V)$ based on $\pd(V')$.
This result immediately suggests that if one could make use of a sequence of subsamples associated with a sequence of values of $f$, then one could get a better approximation than just making use of a single subsampling.
To obtain this result, we introduce in Section~\ref{subsec:merge_module} the concept of stitching two persistence modules together to create a new persistence module.
In Section~\ref{subsec:pointcoud_subsample}, we outline how this can be used to obtain bounds on the persistence diagram of $\cX$ from a sequence of subsamples $\cX = \cX_0 \supset \cX_1\supset \cdots \supset \cX_N$ and the associated persistence diagrams.

It can be argued that for applications, the most difficult task is the construction of the interleaving between the two persistence modules.
However, as we hope the examples of Section~\ref{sec:examples} illustrate, once the interleavings are determined, working with our framework is straightforward.
With this in mind, we include Table~\ref{table:approx} in Section~\ref{sec:comparison}, providing an easily-referenced list of translation maps of generalized interleavings for common approximations to Vietoris-Rips and \u{C}ech filtrations.

\section{Preliminaries}
\label{sec:preliminaries}
In this section, we summarize background material and establish notation for the work we present in this paper.
In Section~\ref{subsec:2-preliminaries-Morphisms-Diagrams}, we recall basic facts about persistence modules, their morphisms, and persistence diagrams.   In Section~\ref{subsec:2-preliminaries-Interleavings} we provide a necessary background for  interleavings of persistence modules. In Section~\ref{subsec:2-preliminaries-GaloisConnections} we give a treatment of  monotone functions and Galois connections, and we define matchings. Section~\ref{subsec:induced_matchings} introduces matchings between persistence diagrams induced by morphisms of persistence modules and recalls the results of Bauer and Lesnick \cite{bauer_lesnick} concerning these matchings.

\subsection{Persistence Modules, Persistence Module Morphisms, and Persistence Diagrams}
\label{subsec:2-preliminaries-Morphisms-Diagrams}
This section provides basic facts about persistence modules (Definition~\ref{defn:persistence_module}).
For alternative treatments, see \cite{bauer_lesnick, carlsson_zomorodian, bubenik_desilva_scott, chazal_desilva}.

\begin{defn}
A persistence module $V$ is \emph{trivial} if $V_t =0$ for all $t \in \R$. 
\end{defn}

\begin{defn}
\label{defn:interval_module}
Let $J \subseteq \overline{\R }$ be a nonempty interval and let $\sk$ denote a field.
The \emph{interval persistence module} $(I_{J}, \varphi_{J})$ is defined by the vector spaces 
 \[
   (I_{J})_t := \left\{
     \begin{array}{ll}
       \sk  \text{ if\;} t \in J,\\
       0  \text{ otherwise,}
     \end{array}
   \right.\\
 \]
 and transition maps
 \[
   \varphi_{J}(s, t) := \left\{
     \begin{array}{ll}
       \text{id}_{\sk}  \text{ if\;} s, t \in J, \\
       0  \text{ otherwise.}
     \end{array}
   \right.
\]
\end{defn}

\begin{defn}
\label{defn:persistence_module_morphism}
Let $(V, \varphi_V)$ and $(W, \varphi_W)$ be persistence modules. A \emph{persistence module morphism} $\phi \colon V \rightarrow W$ is a collection of linear maps $\{\phi_t \colon V_t \rightarrow W_t\}_{t \in \R}$ such that the following diagram commutes for all $s, t \in \R$ with $s \leq t$.
\begin{equation*}
\begin{tikzcd}
V_s \arrow[d, rightarrow,"\phi_s"] \arrow{r}{\varphi_V(s,t)} & V_t \arrow[d, rightarrow,"\phi_t"] \\
W_s \arrow{r}{\varphi_W(s,t)} & W_t
\end{tikzcd}
\end{equation*}
If $\phi_t$ is injective (surjective) for every $t \in \R$, then we say that $\phi$ is a monomorphism (epimorphism). 
A persistence module morphism that is both a monomorphism and an epimorphism is an isomorphism.
\end{defn}

Persistence modules and their morphisms form an abelian category \cite{Bubenik2014}. 
Thus, it makes sense to talk about submodules, quotients, and direct sums of persistence modules. 
Moreover, the kernel and image of a persistence module morphism are submodules, and the cokernel of a persistence module morphism is a quotient persistence module. 
The following fundamental result (see \cite{chazal_desilva, crawley-boevey}) guarantees that  nontrivial PFD persistence modules are direct sums of interval persistence modules.

\begin{thm}
\label{thm:interval_decomp}
Every non-trivial PFD persistence module $V$ is a direct sum of interval persistence modules. Moreover, the direct sum decomposition of $V$ into interval persistence modules is unique up to a reindexing of these interval persistence modules.
\end{thm}

This direct sum decomposition is called the \emph{interval decomposition} of $V$, which we represent using the definitions that follow. 

\begin{defn}

The set $\E$ of \emph{decorated points} is defined by
\[
\E := \R \times \setof{-,+} \cup \setof{-\infty, \infty}.
\]
For $ t \in \R$, define  $t^- := (t,-)$ and $t^+ := (t,+)$.  
Consider the ordering $- < +$ on the set $\setof{-,+}$. 
Then there is a natural ordering on $\E$ induced by a lexicographical ordering of $\overline{\R}$ and $\setof{-,+}$, in that order,  with $\{-\infty\}$ the minimal element and $\{\infty\}$ the maximal element.

\end{defn}

\begin{defn}
\label{defn:decorInt}
 Let $a,b \in \overline{\R}$ such that $a \leq b$. Any nonempty interval $J$ with endpoints $a$ and $b$ can be represented by an ordered pair $(\cB(J),\cD(J))$  of decorated points where:
 \[
 \cB(J) := \left\{
     \begin{array}{ll}
       -\infty  \text{ if\;} a  = -\infty, \\
       a^-  \text{ if\;} J \text{ is left closed,} \\
       a^+   \text{ if\;} J \text{ is left open,}
     \end{array}
   \right.
   \quad \text{and} \quad
   \cD(J) := \left\{
     \begin{array}{ll}
       \infty  \text{ if\;} b  = \infty, \\
       b^-  \text{ if\;} J \text{ is right open,} \\
       b^+   \text{ if\;} J \text{ is right closed.}
     \end{array}
   \right.
\]
For an ordered pair $(d_1,d_2)$ of decorated points with $d_1 < d_2$, we denote  the interval they represent by $\interval{d_1}{d_2}$.
\end{defn}

\begin{defn}
Let $V$ be  a PFD persistence module and $\mathscr{J}_V$ be a multiset of interval persistence modules in the interval decomposition of $V$. Suppose that the function ${m \colon \mathscr{J}_V \to \N}$ assigns to every interval persistence module $I_J \in \mathscr{J}_V$  its multiplicity   in  $\mathscr{J}_V$. The  \emph{persistence diagram} of $V$ is defined as the set
\[
\pdbf(V) :=  \bigcup_{I_J \in \mathscr{J}_V} \setof{ [\cB(J), \cD(J), 1], \ldots, [\cB(J), \cD(J),m(I_J)]} \subset \E\times\E\times\N.
\]
\end{defn}

Note that for every interval persistence module present in the interval decomposition of $V$, there is exactly one point in the persistence diagram. These points can be  totally ordered as in the following definition.

\begin{defn}
Let $\pdbf$ be a persistence diagram. The \emph{ left-handed ordering} of the points $[b,d,i] \in \pdbf$ is given by a lexicographical ordering applied to $(b,-d,i)$, where the minus sign indicates reversing the ordering for the second coordinate. The  \emph{ right-handed ordering} of $\pdbf$ is given by a lexicographical ordering applied to  $(d,b,i)$.
\end{defn}

\subsection{Persistence Module Interleavings}
\label{subsec:2-preliminaries-Interleavings}

In this section we review the notion of persistence module interleavings, introduced by Chazal, et al. in \cite{chazal_proximity} and generalized by Bubenik, et al. in \cite{bubenik_desilva_scott}.
Interleavings provide a measure of similarity  between persistence modules.

\begin{defn}
\label{defn:pretranslation}
A function $\sigma : \R \rightarrow \R$ is \emph{monotone} if $x \leq y$ implies that $\sigma(x) \leq \sigma(y)$. 
If, in addition, $x \leq \sigma(x)$ for all $x \in \R$, then $\sigma$ is called a \emph{translation map}. 
\end{defn}

\begin{defn}
\label{defn:translation_pair}
A pair  $(\tau, \sigma)$ of monotone functions is a \emph{translation pair} if     $\tau \circ \sigma$ and $\sigma \circ \tau$ are translation maps. 
\end{defn}

\begin{defn}
\label{defn:shifted_persistence_module}
Let $\sigma : \R \rightarrow \R$ be monotone and let $(V, \varphi_V)$ be a persistence module.
The \emph{$\sigma$-shifted persistence module} $(V(\sigma), \varphi_{V(\sigma)})$ is defined by the vector spaces
\[
   V(\sigma)_t := V_{\sigma(t)}
\]
for  $t \in \R$ and transition maps
\[
   \varphi_{V(\sigma)}(s, t) := \varphi_V(\sigma(s), \sigma(t))
\]
for every $s \leq t \in \R$.
\end{defn}

\begin{defn}
Let $(V, \varphi_V)$ and $(W, \varphi_W)$ be  persistence modules,  $\phi : V \rightarrow W$ a persistence module morphism, and $\sigma\colon \R \to \R$ a monotone function. 
The \emph{$\sigma$-shifted persistence module morphism} $\phi(\sigma) : V(\sigma) \rightarrow W(\sigma)$ is defined by
\[
\phi(\sigma)_t := \phi_{\sigma(t)}
\]
for every $t \in \R$.
\end{defn}

\begin{defn}
\label{defn:general_interleaving}
Let $(V, \varphi_V)$ and $(W, \varphi_W)$ be persistence modules and let $(\tau, \sigma)$ be a translation pair. The ordered pair of  persistence modules $(V,W)$ is \emph{$(\tau, \sigma)$-interleaved} if there exist persistence module morphisms $\phi\colon V \rightarrow W(\tau)$ and ${\psi \colon W \rightarrow V(\sigma)}$ such that 
\[
\psi(\tau)_t \circ \phi_t = \varphi_V[t, (\sigma \circ \tau)(t)]
\]
and
\[
\phi(\sigma)_t \circ \psi_t = \varphi_W[t, (\tau \circ \sigma)(t)]
\]
for all $t \in \R$.
We refer to these last two conditions as the \emph{commutativity constraint} of the interleaving.
The persistence module morphisms $\phi$ and $\psi$ are called \emph{interleaving morphisms}.
\end{defn}

\begin{defn}
\label{rem:ker_morphism}
Given  a persistence module $V$  and a translation map $\sigma$, define a persistence module morphism $\phi_{\setof{V,\sigma}} \colon V \rightarrow V(\sigma)$ by
 $(\phi_{\setof{V,\sigma}})_t := \varphi_V(t, \sigma(t))$ for all $t \in \R$.
\end{defn}

\begin{rem}
\label{rem:delta_interleaved}
The notion of \emph{$\delta$-interleaved} persistence modules,  presented in \cite{bauer_lesnick, chazal_proximity, chazal_desilva}, is equivalent to the notion of  $(\tau, \sigma)$-interleaved  persistence modules with $\tau(t) = t + \delta = \sigma(t)$. 
\end{rem}

\begin{rem}
\label{rem:isomoprhic_interleaving}
Two persistence modules that are $0$-interleaved are isomorphic as persistence modules.
\end{rem}

Recall that the transition maps of the trivial persistence module are trivial. 
The following definition provides a way of quantifying the similarity between a persistence module $V$ and the trivial persistence module in terms of a translation map.

\begin{defn}
\label{defn:sigma_trivial}
Let $\sigma$ be a translation map. 
A persistence module $(V, \varphi_V)$ is \emph{$\sigma$-trivial} if $\varphi_V(t, \sigma(t)) = 0$ for all $t \in \mathbb{R}$.
\end{defn}

The following proposition provides information about the kernel and cokernel of the interleaving morphisms of two interleaved persistence modules.

\begin{prop}
\label{prop:interleaving_kernel_cokernel}
Let $V$ and $W$ be persistence modules such that 
$(V,W)$ are $(\tau, \sigma)$-interleaved via the morphisms $\phi : V \rightarrow W(\tau)$ and ${\psi : W \rightarrow V(\sigma)}$. 
Then 
\begin{enumerate}[(i)]
\item $\ker \phi$ and $\coker \phi$ are $(\sigma \circ \tau)$-trivial, and
\item $\ker \psi$ and $\coker \psi$ are $(\tau \circ \sigma)$-trivial.
\end{enumerate}
\end{prop}

\begin{proof}
(i) The persistence module $\ker \phi$ is $(\sigma \circ \tau)$-trivial if and only if $\varphi_{\ker \phi}(t, \sigma \circ \tau(t)) = 0$ for all $t \in \R$. 
By the commutativity constraint of a $(\tau, \sigma)$-interleaving, we know that $\varphi_{V}(t, \sigma \circ \tau(t)) = \psi(\tau)_{t} \circ \phi_t$. 
Thus, 
\[
\varphi_V(t, \sigma \circ \tau(t)) |_{\ker \phi_t} = \psi(\tau)_{t} \circ \phi_t |_{\ker \phi_t}  = 0.
\]
By the definition of the persistence module $\ker \phi$, we have 
\[
\varphi_{\ker \phi}(t, \sigma \circ \tau(t)) = \varphi_V (t, \sigma \circ \tau(t)) |_{\ker \phi_t} = 0,
\] 
and so we are done.

The persistence module $\coker \phi$ is $(\sigma \circ \tau)$-trivial if and only if  $\varphi_{\coker \phi}(t, \sigma \circ \tau(t)) = 0$  for all $t \in \R$. 
Recall  that the transition maps of the persistence module $\coker \phi$ are defined to be the unique linear maps $\varphi_{\coker \phi}(r, s)$ such that
\[
\varphi_{\coker \phi}(r, s) \circ q_r = q_s \circ \varphi_{W(\tau)}(r, s)
\]
for every $r \leq s \in \R$, where $q_r := (\alpha \mapsto \alpha + \im \phi_r)$ for every $\alpha \in W(\tau)_r$ is the quotient map.
Thus, it suffices to show that $\im \varphi_{W(\tau)}(t, \sigma \circ \tau(t)) \subseteq \im \phi({\sigma \circ \tau)_t}$ for every $t \in \R$.
For  $t \in \R$ we have
\begin{align*}
\varphi_{W(\tau)}(t, \sigma \circ \tau(t)) 
&= \varphi_{W}(\tau(t), \tau \circ \sigma \circ \tau(t)) \\
&= \phi(\sigma)_{\tau(t)} \circ \psi_{\tau(t)} \\
&= \phi_{\sigma \circ \tau(t)} \circ \psi_{\tau(t)},
\end{align*}
where the first equality follows from the definition of the maps $\varphi_{W(\tau)}$, the second equality  follows from the commutativity constraint of the interleaving morphisms $\phi$ and $\psi$, and the last equality  follows from the definition of $\phi(\sigma)$.
Hence, we have shown that
\[
\im \varphi_{W(\tau)}(t, (\sigma \circ \tau)(t)) \subseteq \im \phi_{\sigma \circ \tau(t)}
\]
for every $t \in \R$.

Part (ii) follows from (i) by reversing the roles of $\phi$ and $\psi$, creating a $(\sigma, \tau)$-interleaving of $W$ and $V$; it follows directly that $\ker \psi$ and $\coker \psi$ are $(\tau \circ \sigma)$-trivial.
\end{proof}

We close this section by recalling a result \cite[Proposition 2.2.11]{bubenik_desilva_scott} that allows us to compose  interleavings.

\begin{prop}
\label{prop:interleaving_composition}
Let $(U, \varphi_U), (V, \varphi_{V}),$ and $(W, \varphi_W)$ be persistence modules such that 
$(U,V)$ are $(\tau, \sigma)$-interleaved  and  $(V,W)$ are $(\tau', \sigma')$-interleaved.  Then the persistence modules $(U,V)$ are ${(\tau' \circ \tau, \sigma \circ \sigma')}$-interleaved.
\end{prop}

\subsection{Galois Connections}
\label{subsec:2-preliminaries-GaloisConnections}

In this section we provide a brief review of Galois connections (see \cite{davey:priestley}) and establish some Galois connections that are used in Section~\ref{sec:algebraic_stability}.

\begin{defn}
\label{defn:galoisconnection}
Let $P$ and $Q$ be posets and suppose $f : P \to Q$ and $g : Q \to P$ are monotone functions. 
The pair $(f,g)$ \emph{is a Galois connection} if  for all $x\in P$ and all $y\in Q$
\[ f(x) \leq y \text{ if and only if } x \leq g(y). \]
\end{defn}

\begin{prop}
\label{prop:galoiscomposition}
Suppose $P$, $Q,$ and $R$ are posets and $f : P \to Q$, $g : Q \to P$, $f' : Q \to R$, and $g' : R \to Q$ are monotone functions. Suppose further that $(f,g)$ and $(f', g')$ are Galois connections. Then $(f' \circ f, g \circ g')$ is a Galois connection.
\end{prop}
\begin{proof}
For all $x \in P$, $y \in R$, we have $(f' \circ f) (x) \leq y \Leftrightarrow f(x) \leq g'(y) \Leftrightarrow x \leq (g\circ g')(y)$. 
\end{proof}

We make use of Galois connections whose definition requires the poset $\R_L$  of lower sets of $\R$ (i.e. intervals $\langle -\infty, e\rangle$ for  $e \in \E$) and the poset $\R_U$ of upper sets  of $\R$ (i.e. intervals $\langle e, \infty\rangle$ for  $e \in \E$). 
In both cases, the ordering is given by inclusion. 
Define the order isomorphisms  $\ket{\cdot} : \E \to \R_L$ and $\bra{\cdot} : \E \to \R_U$ as
\[
\ket{e} := \langle -\infty, e\rangle \qquad \text{and} \qquad \bra{e} := \langle e, \infty\rangle.
\]
Moreover, for any set $S \subseteq \R$,  define
\[
{\downarrow}S = \{ x \in \R : \exists y \in S \; \text{s.t.} \; x \leq y \} \in \R_L,
\]
\[
{\uparrow}S = \{ x \in \R : \exists y \in S \; \text{s.t.} \; y \leq x \} \in \R_U.
\]

\begin{defn}  
\label{defn:promotions}
Let $\sigma : \R \to \R$ be a monotone function. We define $\sigma^\downarrow : \E \to \E$, $\sigma^\uparrow : \E \to \E$, and $\sigma^\star : \E \to \E$  by requiring that the following sets are equal: 
\begin{align*}
\ket{\sigma^\downarrow(e)} &= \, {\downarrow}\{ \sigma(x) : x \in \ket{e} \},\\
\bra{\sigma^\uparrow(e)} &= \, {\uparrow}\{ \sigma(x) : x \in \bra{e} \}, \text{ and}\\
\ket{\sigma^\star(e)} &= \sigma^{-1}(\ket{e}) \text{, or, equivalently, } \bra{\sigma^\star(e)} = \sigma^{-1}(\bra{e}).
\end{align*}
for all $e \in \E$. Note that these functions  are defined since $\ket{\cdot}$ and $\bra{\cdot}$ are order isomorphisms and $\sigma$ is monotone.
\end{defn}

\begin{prop}
\label{prop:promotiondistribution}
Let $\sigma, \tau : \R \to \R$ be monotone functions. Then $(\sigma \circ \tau)^{\uparrow} = \sigma^\uparrow \circ \tau^{\uparrow}$, $(\sigma \circ \tau)^{\downarrow} = \sigma^\downarrow \circ \tau^{\downarrow}$, and $(\sigma \circ \tau)^{\star} = \tau^{\star} \circ \sigma^{\star}$.
\end{prop}
\begin{proof}
It is easy to verify that ${\uparrow} (\sigma \circ \tau)(S)) = {\uparrow} \sigma ( {\uparrow}\tau (S))$, ${\downarrow} (\sigma \circ \tau)(S)) = {\downarrow} \sigma ( {\downarrow}\tau (S))$, and $(\sigma \circ \tau)^{-1}(S) = \tau^{-1}(\sigma^{-1}(S))$ for any $S \subseteq \R$. Now the result follows from Definition~\ref{defn:promotions} and the above equalities  applied to  $S = \bra{e}$ or $S = \ket{e}$ for $e \in \E$.
\end{proof}

\begin{prop}
\label{prop:endpointgaloisconnections}
Let $\sigma : \R \to \R$ be a monotone function. Then both $(\sigma^\downarrow, \sigma^\star)$ and $(\sigma^\star, \sigma^\uparrow)$ are Galois connections. 
\end{prop}

\begin{proof}

We show $(\sigma^\downarrow, \sigma^\star)$ is a Galois connection. Suppose first that $\sigma^{\downarrow}(x) \leq y$ for some $x,y \in \E$. We show $x \leq \sigma^\star(y)$. Since $\ket{\cdot}$ is an order isomorphism, $\sigma^{\downarrow}(x) \leq y$ is equivalent to $\ket{\sigma^{\downarrow}(x)} \subseteq \ket{y}$. By the definition of $\sigma^{\downarrow}$, this is equivalent to ${\downarrow}\sigma(\ket{x}) \subseteq \ket{y}$. Taking the preimage of both sides yields $\sigma^{-1}({\downarrow}\sigma(\ket{x})) \subseteq \sigma^{-1}(\ket{y})$. Since $\ket{x} \subseteq \sigma^{-1}(\sigma(\ket{x})) \subseteq \sigma^{-1}({\downarrow}\sigma(\ket{x}))$, we obtain $\ket{x} \subseteq \sigma^{-1}(\ket{y})$. Since $\ket{\cdot}$ is an order isomorphism, we conclude that $x \leq \sigma^\star(y)$. 

We now prove the converse. That is, we suppose that $x \leq \sigma^\star(y)$ and show $\sigma^{\downarrow}(x) \leq y$. From $x \leq \sigma^\star(y)$, we obtain $\ket{x} \subseteq \sigma^{-1}(\ket{y})$. Applying $\sigma$ to both sides and taking the downward closure gives $ {\downarrow} \sigma(\ket{x}) \subseteq {\downarrow} \sigma(\sigma^{-1}(\ket{y}))$. See that ${\downarrow} \sigma(\sigma^{-1}(\ket{y})) = \ket{y}$, hence $ {\downarrow} \sigma(\ket{x}) \subseteq \ket{y}$, or equivalently, $\sigma^{\downarrow}(x) \leq y$, as desired. Hence, the pair $(\sigma^\downarrow, \sigma^\star)$ is a Galois connection. To show the pair $(\sigma^\star, \sigma^\uparrow)$ is a Galois connection, one proceeds similarly.
\end{proof}

The following maps are used to move between  points in $\overline{\R}$ and  decorated points in $\E$.

\begin{defn}
\label{def::EtoRmaps}
 The maps $\pi \colon \E \to \overline{\R}$, $i^- : \overline{\R} \to \E$ and $i^+ : \overline{\R} \to \E$ are defined by:
\[
{\pi(t^ \pm) = t} \text{ , } i^\pm(t) = t^\pm, 
\]
for $ t \in \R$, and
\[
\pi(\pm\infty) = \pm\infty \text{ , } i^\pm(\pm\infty) = \pm\infty.
\]
\end{defn}

\begin{defn}
\label{defn:demotions}
Let $f : \E \to \E$ be a monotone function.
Define functions  $f_+ \colon \overline{\R} \to \overline{\R}$  and $f_- \colon \overline{\R} \to \overline{\R}$ via
\[
f_+ := \pi \circ f \circ i^+\text{ and } f_- := \pi \circ f \circ i^- . 
\]
\end{defn}

We close this section by establishing some Galois connections that will be needed later.
\begin{prop}
\label{prop:galoisinclusionprojection}
Both $(i^-, \pi)$ and $(\pi, i^+)$ are Galois connections.
\end{prop}
\begin{proof}
First we show that $(\pi, i^+)$ is a Galois connection.
Using the easily-verified relations $\pi \circ i^+ = \text{id}$ and $\text{id} \leq i^+ \circ \pi$, we have, for all $x \in \E$ and $y \in \overline{\R}$, the circle of implications $(\pi(x) \leq y) \Rightarrow ((i^+ \circ \pi) (x) \leq i^+ (y)) \Rightarrow ( x \leq i^+(y) ) \Rightarrow (\pi(x) \leq (\pi \circ i^+)(y)) \Rightarrow (\pi(x) \leq y)$. Thus, $(\pi(x) \leq y) \Leftrightarrow (x \leq i^+(y))$. That is, the pair $(\pi, i^+)$ is a Galois connection.

Showing that the pair $(i^-, \pi)$ is a Galois connection proceeds similarly. Using the easily-verified relations $\pi \circ i^- = \text{id}$ and $i^- \circ \pi \leq \text{id}$, we have,  for all $x \in \overline{\R}$ and $y \in \E$, the circle of implications $(x \leq \pi(y)) \Rightarrow (i^- (x) \leq (i^- \circ \pi) (y) ) \Rightarrow (i^-(x) \leq y) \Rightarrow ((\pi \circ i^-)(x) \leq \pi(y)) \Rightarrow ( x \leq \pi(y))$.
Thus, $(x \leq \pi(y)) \Leftrightarrow (i^-(x) \leq y)$. That is, the pair $(i^-, \pi)$ is a Galois connection.
\end{proof}

\begin{prop}
\label{prop:demotionconnection}
Suppose $f, g : \E \to \E$ are monotone functions such that the pair $(f,g)$ is a Galois connection. Then the pair $(f_-, g_+)$ is a Galois connection.
\end{prop}

\begin{proof}
By Definition~\ref{defn:demotions}, we have $(f_-, g_+) = (\pi \circ f \circ i^-,  \pi \circ g \circ i^+)$.
The result  follows from Proposition~\ref{prop:galoiscomposition} and Proposition~\ref{prop:galoisinclusionprojection}. 
\end{proof}

\subsection{Induced Matchings on Persistence Diagrams}
\label{subsec:induced_matchings}

In this section, we summarize the work of  Bauer and Lesnick~\cite{bauer_lesnick, bauer_lesnick_categorified} on matchings of persistence diagrams of PDF persistence modules $V$ and $W$ induced by a morphism  $\phi : V \rightarrow W$. 

\begin{defn}
\label{defn:matching}
Let $\X$ be a relation between sets $S$ and $T$ (i.e. $\X \subseteq S \times T$). We say that $\X$ is a \emph{matching} $\X : S \matching T$  if $\X$ is the graph of an injective function $\X' : S' \to T'$, where $S' \subseteq S$ and $T' \subseteq T$. We define the domain and image of a matching via $\dom \X := \dom \X'$ and $\im \X := \im \X'$, and we will use the notation $\X(s) = t$ to denote $(s,t) \in \X$. 
\end{defn}

We use the following notation to define matchings induced by morphisms.

\begin{defn}
\label{defn:barcode_orderings}
Let $(V,\varphi_V)$ be a PDF persistence module. For  $b, d \in \mathbb{E}$, we define two subsets of the persistence diagram $\pdbf(V)$ by:
\[
\pdbf_b(V) := \setof{ [b, d',i] \colon [b, d',i] \in \pdbf(V) },
\]
\[
\pdbf^d(V) := \setof{ [b', d,i] \colon [b', d,i] \in \pdbf(V) }.
\]
\end{defn}

If $V$  is a PFD persistence module, then the sets $\pdbf_b(V)$ and $\pdbf^d(V)$ are countable for every $b,d \in \E$. The left-handed (right-handed) ordering on $\pdbf(V)$ induces a total ordering on $\pdbf_b(V)$ ($\pdbf^d$(V)). We will always consider these sets together with these induced orderings. Therefore, if we talk about the first $n$ points in $\pdbf_b(V)$ or $\pdbf^d(V)$, we mean  the $n$ smallest points with respect to the induced ordering. The following proposition allows us to define matchings between the  PFD persistence modules as introduced in~\cite{bauer_lesnick}.

\begin{prop}[Theorem 4.2~\cite{bauer_lesnick}]
Let $V$ and $W$ be PFD persistence modules, and let the symbol $|\cdot|$ denote the cardinality of a set.
\begin{enumerate}[(i)]
\item If there exists a monomorphism  $V   \hookrightarrow W$, then  $|\pdbf^d(V)| \leq |\pdbf^d(W)|$ for $d\in \E$.
\item If  there exists an epimorphism $V  \twoheadrightarrow W$, then  $|\pdbf_b(W)| \leq |\pdbf_b(V)|$ for $b\in \E$.
\end{enumerate}
\end{prop}

The next two propositions establish the matchings induced by monomorphisms and epimorphisms.  
The proof of parts (i)-(iii) of each proposition is a simple consequence of the previous proposition, while (iv) follows from \cite[Theorem 4.2]{bauer_lesnick}.

\begin{prop}
\label{prop:match_i}
Let $V, W$ be PFD persistence modules.
If there exists a monomorphism from $V$ to $W$, then there exists a unique matching $\cX_{i (V,W)} \colon \pdbf(V)  \matching \pdbf(W) $ which satisfies: 
\begin{enumerate}[(i)]
\item the domain of $\cX_{i(V,W)}$ is $\pdbf(V)$,
\item $\cX_{i(V,W)}$ preserves the right-handed ordering,
\item $\cX_{i(V,W)}$ maps the points in $\pdbf^d(V)$ to  the smallest $|\pdbf^d(V)|$  points in $\pdbf^d(W)$, 
\item if $\cX_{i(V,W)}([b,d,i]) = [b',d',i']$, then $d=d'$ and $b'\leq b$.
\end{enumerate}
\end{prop}

\begin{prop}
\label{prop:match_p}
Let $V, W$ be PFD persistence modules.
If there exists an epiomorphism  from $V$ to  $W$, then there exist a unique matching $\cX_{s(V,W)} \colon \pdbf(V)  \matching \pdbf(W) $ that satisfies:  
\begin{enumerate}[(i)]
\item the image of $\cX_{s(V,W)}$ is $\pdbf(W)$,
\item the inverse relation  $\cX^{-1}_{s(V,W)}$ preserves the left-handed ordering, 
\item $\cX^{-1}_{s(V,W)}$ maps the points in $\pdbf_b(W)$ to the smallest  $|\pdbf_b(W)|$  points in $\pdbf_b(V)$, 
\item if $\cX_{s(V,W)}([b,d,i]) = [b',d',i']$ then $b=b'$ and $d'\leq d$.
\end{enumerate}
\end{prop}

Every persistence module morphism $\phi \colon V \to W$ can be factored as the composition of an epimorphism and monomorphism as follows:
\begin{displaymath}
\begin{tikzcd}
V  \twoheadrightarrow \im \phi \hookrightarrow W.
\end{tikzcd}
\end{displaymath}
Therefore, we can define a matching $\cX_\phi \colon \pdbf(V)  \matching \pdbf(W)$ via the composition of the following relations:
\[
\cX_\phi := \cX_{s(V,\im \phi)} \circ \cX_{i (\im \phi,W)}.
\]

In general, it is not true that if  $\phi \colon U \rightarrow V$ and $\psi \colon V \rightarrow W$ are PFD persistence module morphisms then $\X_{\psi \circ \phi} = \X_\psi \circ \X_\phi$.
However, the following result \cite[Proposition 5.7]{bauer_lesnick} provides hypotheses under which this is true.
\begin{prop}
\label{prop:functoriality}
Let $\phi : U \rightarrow V$ and $\psi : V \rightarrow W$ be PFD persistence module morphisms. If $\phi$ and $\psi$ are either both injective or both surjective, then $\X_{\psi \circ \phi} = \X_\psi \circ \X_\phi$.
\end{prop}

\begin{defn}
\label{defn:bounds}
Let $A, B \subseteq \R$.
We say that \emph{$A$ bounds $B$ below}, written $A \bb B$,  if for all $y \in B$, there exists some $x \in A$ with $x \leq y$.
We say that \emph{$B$ bounds $A$ above}, written $A \ba B$, if for all $x \in A$, there exists some $y \in B$ such that $x \leq y$. 
We say that \emph{$B$ overlaps $A$ above}\emph{$B$ overlaps $A$ above}, written $A \overlap B$, if and only if each of the following conditions hold:
\[
A \bb B, \qquad
A \ba B, \qquad \text{and} \qquad
A \cap B \neq \emptyset.
\]
\end{defn}

\begin{prop}\emph{\cite[Proposition 5.3]{bauer_lesnick}}
\label{prop:admissible}
Let $\phi \colon V \rightarrow W$ be a PFD persistence module morphism. 
If $\cX_\phi([b,d,i]) = [b',d',i']$, then  $\interval{b'}{d'}\overlap\interval{b}{d} $.
\end{prop}

\section{Generalized Induced Matching Theorem}
\label{sec:gen_induced_matching}
In this section we present a generalization of the Induced Matching Theorem of \cite{bauer_lesnick}.

\begin{defn} 
Let $\sigma$ be a translation map  and let $b,d\in \E$. 
An interval   $\interval{b}{d}$ is \emph{$\sigma$-trivial} if $\interval{b}{d} \cap  \sigma(\interval{b}{d}) = \emptyset$.  
A point $[b,d,i] \in \E\times\E\times\N$  is  $\sigma$-trivial if   $\interval{b}{d}$ is \emph{$\sigma$-trivial}. 
A point in $\E\times\E\times\N$ that is not $\sigma$-trivial is called \emph{$\sigma$-nontrivial}.
\end{defn}

\begin{thm}
\label{thm:geninducedmatching}
Let $\phi \colon V \rightarrow W$ be a PDF persistence module morphism and  $\sigma$ a translation map. 
Suppose that  $\cX_\phi([b,d,i]) = [b',d',i']$.  
\begin{enumerate}[(i)]
\item If $\coker \phi$ is $\sigma$-trivial, then
\[
\interval{b}{d} \bb \sigma(\interval{b'}{d'}) \text{ and } \im \cX_\phi \text{ contains all } \sigma\text{-nontrivial points in } \pdbf(W).
\]
\item  If $\ker \phi$ is $\sigma$-trivial, then 
\[
\sigma^{-1}(\interval{b}{d} )\ba \interval{b'}{d'} \text{ and } \dom \cX_\phi \text{ contains all } \sigma\text{-nontrivial points in } \pdbf(V). 
\] 
\end{enumerate}
\end{thm}

Note that the Induced Matching Theorem of \cite{bauer_lesnick} follows from Theorem~\ref{thm:geninducedmatching} by setting $\sigma(t) = t + \delta$ for $\delta \geq 0$.

The remainder of this section is devoted to the proof of Theorem~\ref{thm:geninducedmatching}.

\begin{defn}
\label{defn:w_sigma}
Let $V$ be a persistence module and  $\sigma$ a translation map. 
Define  vector spaces of the  persistence module $V^\sigma$ by
\[
V_t^\sigma := \bigcup_{\setof{ s \colon \sigma(s) \leq t}.
} \im \varphi_V(s, t)
\]
for $t\in \R$. 
The linear maps $\varphi_{V^\sigma}$ are given by restriction of the maps $\varphi_V$ to $V^\sigma$.
\end{defn}

The following lemma shows that $V^\sigma$ is a persistence module.

\begin{lem}
\label{lem:w_sigma_submodule}
Let $V$ be a persistence module and  $\sigma$ a translation map. 
Then, $V^\sigma$ is a persistence submodule of $V$.
\end{lem}

\begin{proof}
By definition, $V_t^\sigma$ is a subspace of $V_t$ for  $t\in \R$.
To see that $V^\sigma$ is a persistence submodule of $V$, we must show that $\im \varphi_V(s, t)|_{V_s^\sigma} \subseteq V_t^\sigma$ for all $s \leq t$. 
To  do this, we consider $y \in V^\sigma_s$ and show  that $\varphi(s,t)(y) \in V^\sigma_t$.
By definition, there exists $x \in V_{t'}$ for some  $t'\in \R$ such that $ \sigma(t') \leq s$ and   $\varphi_V(t',s)(x) = y$. Thus,
\[
\varphi_V(s, t)(y) =\varphi_V(s,t)[\varphi_V(t',s)(x)] = \varphi_V(t',t)(x).
\]
Since $\sigma$ is a translation map,  $t' \leq \sigma(t') \leq s \leq t$, and so  $\varphi(s,t)(y) \in \im \varphi_V(t',t) \subseteq V_t^\sigma$.
\end{proof}

\begin{lem}
\label{lem:optimality_w_k}
Let $\phi : V \rightarrow W$ be a persistence module morphism and  ${\sigma}$ a translation map.
\begin{enumerate}[(i)]
\item If $\im \phi$ is $\sigma$-trivial, then $W_t^\sigma \subseteq \im \phi_t \subseteq W_t$ for every $t \in \R$, and
\item if $\ker \phi$ is $\sigma$-trivial, then $\ker \phi_t \subseteq (\ker \phi_{\setof{V,\sigma}})_t \subseteq V_t$ for every $t \in \R$.
\end{enumerate}
\end{lem}

\begin{proof}

(i) By definition, given a morphism $\phi : V \rightarrow W$, the persistence module $\coker{\phi}$ is $\sigma$-trivial if and only if 
\[
{\varphi_{\coker \phi}(t , \sigma(t)) = 0} \quad \text{for all} \quad t \in \R, 
\]
which is true if and only if 
\[
\im \varphi_W(t , \sigma(t)) \subseteq \im \phi(\sigma)_t \quad \text{for all} \quad t \in \R,
\]
which again is true if and only if for each $t \in \R$ and each $x \in W_t$, there exists some $y \in V_{\sigma(t)}$ such that
\[
\varphi_W(t , \sigma(t))(x) = \phi(\sigma)_t(y).
\]
So, to prove that  $W^\sigma_t \subseteq \im \phi_t$, it is enough to show that $\im \varphi_W(t' , t) \subseteq \im \phi_{t'}$ for every $t'\in \R$ such that $t' \leq \sigma(t)$. By commutativity of the diagram 
\[
\begin{tikzcd}
&&& V_{\sigma(t')} \arrow[rrr,rightarrow,"\varphi_V(\sigma(t')\text{,} t)"]  \arrow[d,"\phi(\sigma)_{t'}"] &&& V_t  \arrow[d,"\phi_t"] \\
W_{t'} \arrow[rrr, "\varphi_W(t'\text{,} \sigma(t'))"] &&& W_{\sigma(t')} \arrow[rrr, "\varphi_W(\sigma(t')\text{,} t)"] &&& W_t \\
\end{tikzcd}
\]
we have
\[
\varphi_W(t', t)(x) = \varphi_W(\sigma (t'), t) [ \varphi_W(t', \sigma(t'))(x)] = \phi_t [\varphi_V(\sigma(t'), t)(y)],
\]
and so  $\im \varphi_W(t' , t) \subseteq \im \phi_t$.

To prove (ii), we  show that  $\ker \phi_t \subseteq (\ker \varphi_{\setof{V,\sigma}})_t$ for all $t \in \R$ whenever $\phi$ has a $\sigma$-trivial kernel.
By definition, $\ker \phi$ is $\sigma$-trivial if and only if 
\[
\varphi_V(t, \sigma(t))|_{\ker \phi_t} = \varphi_{\ker \phi}(t, \sigma(t)) = 0
\]
for all $t \in \R$.
Hence, $\ker \phi_t \subseteq \ker \varphi_V(t, \sigma(t)) = (\ker \phi_{\setof{V,\sigma}})_t$ for all $t \in \R$ if and only if $\ker \phi$ is $\sigma$-trivial.
\end{proof}

We now study the relationship between  the persistence module $V$ and the persistence modules $V^\sigma$ and $V/\ker \phi_{\setof{V,\sigma}}$. We start by considering an interval persistence module.

\begin{lem}
\label{lem:w_sigma_barcode_interval}
Let $I_J$ be an interval persistence module and  $\sigma$ a translation map. 
If $J \cap \sigma(J) \neq \emptyset$, then 
\begin{enumerate}[(i)]
\item $I_J^\sigma \cong I_{J \cap \conv{\sigma(J)}} $, where $\conv{\sigma(J)}$ is the convex hull of $\sigma(J)$,  
\item $V/\ker \phi_{\setof{V,\sigma}} \cong I_{J \cap \sigma^{-1}(J)}$,
\end{enumerate}
If $J \cap \sigma(J) = \emptyset$,  then both persistence modules are trivial.  
\end{lem}

\begin{proof}
(i) We first show that $(I_J^\sigma)_t \neq \emptyset$ for $t \in J \cap \conv{\sigma(J)}$. Since $\sigma$ is a translation map, for every $t \in \conv{\sigma(J)}$, there exist $s\in J$ such that $s \leq \sigma(s) \leq t$.  
The fact that $t\in J$ implies  im $\varphi_{I_J}(s, t) =\text{id}_{\sk}$ and hence $\emptyset \neq \im \varphi_{I_J}(s, t) \subseteq (I_J^\sigma)_t$. 

To finish the proof of (i), we need to show that $(I_J^\sigma)_t = 0$ if $t \notin J \cap \conv{\sigma(J)}$. Suppose that $t \not\in J$. 
Then $(I_J^\sigma)_t \subseteq (I_J)_t =0$. On the other hand, if $t \notin  \conv{\sigma(J)}$ and $\sigma(s) \leq t$, then  $s \not\in J$, and so  $\varphi_{I_J}(s, t)  = 0$ for all $s$ such that $\sigma(s) \leq t$.  

If $J \cap \sigma(J) = \emptyset$, then $J \cap \conv{\sigma(J)} = 0$ and we showed above that $\varphi_{I_J}(s, t)  = 0$ for all $s \leq t$. It follows that $I_J^\sigma$ is trivial.  Similar arguments can be used to prove (ii), and we leave it to the reader. 
\end{proof}

For the following two definitions, for an interval $J \subseteq \R$, we recall the symbols $\cB(J)$ and $\cD(J)$ (Definition~\ref{defn:decorInt})  give the left and right (decorated) endpoints of $J$, respectively.

\begin{prop}
\label{prop:w_sigma_barcode}
Let $V$ be a PFD persistence module and $\sigma$ a translation map. 
Suppose that $[b,d,i] \in \pdbf(V)$ and the interval $J = \interval{b}{d} \cap  \conv{\sigma\interval{b}{d}}$.  
Then  $[\cB(J),d,i] \in \pdbf(V^\sigma)$ if and only if $J \neq \emptyset$.
Moreover,
\[
\cX_{i(V^\sigma,V)}([\cB(J),d,i]) = [b,d,i].
\]
\end{prop}

\begin{proof}

Let $[b,d,i] \in \pdbf(V)$. Since $\sigma$ a translation map,  $J = \interval{\cB(J)}{d}$. 
By Lemma~\ref{lem:w_sigma_barcode_interval}, the interval persistence module $I^\sigma _{\interval{b}{d}}$ is nontrivial  if and only if $J \neq\emptyset$. 
It follows from  Theorem~\ref{thm:interval_decomp} that the interval persistence module $I^\sigma _{\interval{b}{d}}   $($\cong I_{\interval{\cB(J)}{d}}$) belongs to  the interval decomposition of $V^\sigma$ if and only if $J \neq\emptyset$. Thus, $[\cB(J),d,i] \in \pdbf(V^\sigma)$ for every $[b,d,i] \in \pdbf(V)$ such that $\interval{b}{d} \cap  \conv{\sigma\interval{b}{d}} \neq \emptyset$. Now, $\cX_{i(V^\sigma,V)}([\cB(J),d,i]) = [b,d,i]$ is a simple consequence of  Propostion~\ref{prop:match_i}.
\end{proof}

By using Lemma~\ref{lem:w_sigma_barcode_interval}(ii) and similar reasoning as above, one can  prove the following about the matching $\cX_{s(V, V/\ker \phi_{\setof{V,\sigma}})}$.

\begin{prop}
\label{prop:k_kernel_barcode}
Let $V$ be a PFD persistence module and $\sigma$ be a translation map. Suppose that  $[b,d,i] \in \pdbf(V)$. 
Then $[b,\cD(J),i] \in \pdbf(V/\ker \phi_{\setof{V,\sigma}})$ if  and only if $J := \interval{b}{d} \cap  \sigma^{-1}(\interval{b}{d}) \neq \emptyset$.  Moreover,
 \[
 \cX_{s(V, V/\ker \phi_{\setof{V,\sigma}})}([b,d,i]) = [b,\cD(J),i].
 \]
\end{prop}

\begin{proof}[Proof of Theorem~\ref{thm:geninducedmatching}] Our proof closely follows the proof of the Induced Matching Theorem in \cite{bauer_lesnick}.  
To prove (i), we start by establishing the existence of certain matchings. By Lemma~\ref{lem:optimality_w_k}(i), $W^\sigma$ is a submodule of $\im \phi$ and so the matching $\cX_{i(W,\im \phi)}$ is defined. It follows from Lemma~\ref{lem:w_sigma_submodule} that  $W^\sigma$ is a submodule of  $W$ and thus $\cX_{i(W^\sigma,W)}$ is defined. 
Proposition~\ref{prop:functoriality} implies that the following diagram commutes.

\begin{center}
\setlength\mathsurround{0pt}
\begin{tikzcd}
& \pdbf(W) \arrow[dd, strike thru, hookleftarrow,"{\X_{i(\im \phi, W)}}"] \arrow [dr, strike thru, leftarrow, "\X_\phi"] & \\
\pdbf({W^\sigma}) \arrow[ur, strike thru, "{\X_{i(W^\sigma,W)}}", hookrightarrow] & & \pdbf({V}) \arrow[dl, strike thru, twoheadrightarrow, "{\X_{s(V,\im \phi)}}"] \\
& \pdbf({\im \phi}) \arrow[ul, strike thru, "{\X_{i(W^\sigma,\im \phi)}}", hookleftarrow] & \\
\end{tikzcd}
\end{center}

It follows from Proposition~\ref{prop:match_p}(i) that $\im\X_\phi = \im \X_{i(\im \phi, W)}$.  
By commutativity of the left triangle, $\im \X_{i(W^\sigma,W)} \subseteq \im \X_{i(\im \phi, W)}$. 
Now it follows from Proposition~\ref{prop:w_sigma_barcode} that $\im\X_\phi$ contains all $\sigma$-nontrivial points in $\pdbf(W)$.

To finish the proof of (i), we must show that if  $\X_\phi([b,d,i]) = [b',d',i']$,   then  $\interval{b}{d} \bb \sigma(\interval{b'}{d'})$. First we suppose that $[b',d',i']$ is $\sigma$-nontrivial.  In this case for  $J := \interval{b'}{d'} \cap  \conv{\sigma\interval{b'}{d'}}$,  the  point $([\cB(J),d',i'] \in \pdbf(W^\sigma)$  and $\X_{i(W^\sigma, W)}([\cB(J),d',i']) = [b',d',i']$. 
Since $\sigma$ is a translation map, we have $\interval{\cB(J)}{d'} \bb \sigma(\interval{b'}{d'})$. Due to commutativity of the above diagram, $\X_{i(W^\sigma,\im \phi)}([\cB(J),d',i']) = \cX_{s(V,\im \phi)}([b,d,i])$. It follows from Proposition~\ref{prop:match_i}(iv) and  Proposition~\ref{prop:match_p}(iv) that $b \leq \cB(J)$. Therefore,  $\interval{b}{d} \bb \sigma(\interval{b'}{d'})$.
 
On the other hand, if $[b',d',j]$ is $\sigma$-trivial, then $\interval{b'}{d'} \cap {\sigma(\interval{b'}{d'})} = \emptyset$ and every point in $\sigma(\interval{b'}{d'})$ is larger than $d'$. By Proposition~\ref{prop:admissible}, $b < d'$ and we get  $\interval{b}{d} \bb \sigma(\interval{b'}{d'})$.

The proof of (ii) is based on similar ideas combined with the commutativity of the diagram
\begin{center}
\setlength\mathsurround{0pt}
\begin{tikzcd}
& \pdbf(V) \arrow[dd, strike thru, twoheadrightarrow,"{\cX_{s(V,\im \phi)}}"] \arrow [dr, strike thru, rightarrow, "\X_\phi"] & \\
\pdbf({V/\ker \phi_{\setof{V,\sigma}}}) \arrow[ur, strike thru, "{\cX_{s(V,V/\ker \phi_{\setof{V,\sigma}})}}", twoheadleftarrow]  & & \pdbf({W}) \arrow[dl, strike thru, hookleftarrow, "{\X_{i(\im \phi, W)}}"] \\
& \pdbf({\im \phi}) \arrow[ul, strike thru, "{\cX_{s(\im \phi,V/\ker \phi_{\setof{V,\sigma}})}} ", twoheadrightarrow] & \\
\end{tikzcd}
\end{center}
and is left to the reader. 
\end{proof}

\section{Algebraic Stability Theorem for Generalized Interleavings}
\label{sec:algebraic_stability}
\label{sec::Stbility}
In this section we provide a generalization  of the Algebraic Stability Theorem of \cite{bauer_lesnick}. 

\begin{thm}
\label{thm:genalgstability}
Let $(V, \varphi_V)$ and $(W, \varphi_W)$ be PFD persistence modules such that 
$(V,W)$ are $(\tau, \sigma)$-interleaved via the morphisms $\phi : V \rightarrow W(\tau)$ and ${\psi : W \rightarrow V(\sigma)}$. 
There exists a matching $\X \colon \pdbf(V) \matching \pdbf(W)$ such that $\X([b,d,i]) = [b',d',i']$ implies 
\begin{equation}
(\sigma \circ \tau)^{-1}(\interval{b}{d}) \ba \tau^{-1}(\interval{b'}{d'}) \overlap \interval{b}{d} \bb \sigma \circ \tau \circ \tau^{-1}(\interval{b'}{d'}) \label{eq:ineq1}.
\end{equation}
Moreover, if $[b,d,i] \in \pdbf(V)$ is  unmatched, then it is $(\sigma \circ \tau)$-trivial, and if $[b',d',i'] \in \pdbf(W)$ is unmatched, then either $\interval{b'}{d'}\cap \im \tau = \emptyset$ or $\tau^{-1}(\interval{b'}{d'})$ is $(\sigma \circ \tau)$-trivial.
\end{thm}

As in the case of the Induced Matching Theorem, the result of \cite{bauer_lesnick} follows from setting $\tau(t) = \sigma(t) = t + \delta $ for $\delta \geq 0$.

\begin{proof}
It follows from Propostion~\ref{prop:interleaving_kernel_cokernel}(i) that $\ker \phi$ and $\coker \phi$ are $(\sigma \circ \tau)$-trivial. By Theorem~\ref{thm:geninducedmatching}, the domain of  $\X_\phi \colon \pdbf(V) \matching \pdbf(W(\tau))$ contains all $\sigma\text{-nontrivial points}$ in $\pdbf(V)$, and its image contains all $\sigma\text{-nontrivial points}$ in $\pdbf(W(\tau))$. Now suppose that $\X_\phi([b,d,i]) = [x,y,j]$. Then 
\begin{equation}\label{eqn::matchXphi}
(\sigma \circ \tau)^{-1}(\interval{b}{d}) \ba \interval{x}{y} \overlap \interval{b}{d} \bb \sigma \circ \tau \interval{x}{y},
\end{equation}
where the first and the last relations follow from Theorem~\ref{thm:geninducedmatching} and the middle one from Proposition~\ref{prop:admissible}.

To finish the proof, we build an appropriate matching $\X' \colon \pdbf(W(\tau)) \matching \pdbf(W)$. It follows from the definition of $W(\tau)$ that there is a one-to-one correspondence between the  points in $\pdbf(W(\tau))$ and the set $\setof{[b',d',i'] \in \pdbf(W) \colon \interval{b'}{d'} \cap \im \tau \neq \emptyset}$. This correspondence can be realized by a matching $\X' \colon \pdbf(W(\tau)) \matching \pdbf(W)$ such that $\X'([x,y,j]) = [b',d',i']$ implies $\interval{x}{y} = \tau^{-1}(\interval{b'}{d'})$. Thus, the desired matching $\X$ is obtained by the composition  $\X' \circ \X_\phi$. Condition~(\ref{eq:ineq1}) follows from (\ref{eqn::matchXphi}) and the fact that $\interval{x}{y} = \tau^{-1}(\interval{b'}{d'})$.
\end{proof}

Statement~\eqref{eq:ineq1}  in Theorem~\ref{thm:genalgstability} may seem impractical. However, it can be rewritten as a set of inequalities concerning the endpoints of the intervals,  and if the translation maps $\tau$ and $\sigma$ are bijective, then the inequalities can be  simplified considerably. 

\begin{cor}
\label{cor:genalgstab_invertible}
Let $(V, \varphi_V)$ and $(W, \varphi_W)$ be persistence modules such that  $(V,W)$ are $(\tau, \sigma)$-interleaved.  Suppose that $\X([b,d,i]) = [b',d',i']$, where $\X$ is the matching given by Theorem~\ref{thm:genalgstability}.  Condition~\eqref{eq:ineq1} implies that the inequalities: 
\[
\tau^\star (b')  \leq b, \quad \tau^\star (d')  \leq d,
\]
\[ 
\sigma^{\star}(b) \leq \tau^\uparrow \circ \tau^\star(b'), \quad \sigma^\star(d) \leq \tau^\uparrow \circ \tau^\star (d')
\]
hold. Moreover, if the maps $\tau$ and $\sigma$ are bijections, then the above inequalities reduce to 
\[
\tau^\star (b') \leq b, \quad \tau^\star(d') \leq d,
\]
\[ 
\sigma^\star (b') \leq b, \quad \sigma^\star(d') \leq d,
\]
and we have that if $[b,d,i] \in \pdbf(V)$ is  unmatched, then it is $(\sigma \circ \tau)$-trivial, and if $[b',d',i'] \in \pdbf(W)$ is unmatched, then it is $(\tau \circ \sigma)$-trivial.
\end{cor}

\begin{proof}
We start by showing that the first set of inequalities follows from Condition~\eqref{eq:ineq1}. The relation  $\tau^{-1}(\interval{b'}{d'}) \overlap \interval{b}{d}$ together with the monotonicity of $\tau$ implies that $\interval{b'}{d'}\overlap \tau(\interval{b}{d})$, and by definition  we thus have
\[
\interval{b'}{d'} \bb \tau(\interval{b}{d}) \text{ and } \interval{b'}{d'} \ba \tau(\interval{b}{d}).
\]
The inequality $\tau^\star (b')  \leq b$ ($\tau^\star (d')  \leq d$) follows from definition of $\tau^\star$, and the first (second) relation stated above. Now we prove that $\sigma^{\star}(b) \leq \tau^\uparrow \circ \tau^\star(b')$. Starting from the relation $\interval{b}{d} \bb \sigma \circ \tau \circ \tau^{-1}(\interval{b'}{d'})$, we obtain that 
\[
\interval{b}{d} \bb {\uparrow}\setof{\sigma \circ \tau \circ \tau^{-1}(\interval{b'}{d'})}  = \langle (\sigma\circ\tau)^\uparrow \circ \tau^{\star}(b'), \infty \rangle.
\] 
Hence $b \leq (\sigma\circ\tau)^\uparrow \circ \tau^{\star}(b')$, and the desired inequality follows from the definition of $\sigma^\star$. The last inequality follows again from the definitions of $\sigma^{\star}, \tau^\uparrow,$ and $\tau^\star$, and the relation $(\sigma \circ \tau)^{-1}(\interval{b}{d}) \bb \tau^{-1}(\interval{b'}{d'})$ is guaranteed by Condition~\eqref{eq:ineq1}. 

Now the moreover part. We suppose that $\tau$ and $\sigma$ are invertible. Combining the invertibility of $\tau$ with the definitions of $\tau^{\uparrow}$ and $\tau^{\star}$ (i.e. Definition~\ref{defn:promotions}), it is routine to verify  $\tau^{\uparrow}\circ \tau^\star = \mbox{id}$. Similarly, $\sigma^{\uparrow} \circ \sigma^{\star} = \mbox{id}$. The second set of inequalities follows. The statement that if $[b,d,i] \in \pdbf(V)$ is  unmatched, then it is $(\sigma \circ \tau)$-trivial carries over from Theorem~\ref{thm:genalgstability}. The statement that if $[b',d',i'] \in \pdbf(W)$ is unmatched, then it is $(\tau \circ \sigma)$-trivial follows from Theorem~\ref{thm:genalgstability} as well, observing that the $\interval{b'}{d'}\cap \im \tau = \emptyset$ case cannot happen when $\tau$ is surjective, and that the condition that $\tau^{-1}(\interval{b'}{d'})$ is $(\sigma \circ \tau)$-trivial is equivalent to $\interval{b'}{d'}$ being $(\tau \circ \sigma)$-trivial when $\tau$ is injective. 
\end{proof}

Most of the algorithms for computing persistence diagrams   do not store information about decorations of the endpoints, and they produce undecorated persistence diagrams. To define undecorated persistence diagrams we will use the  maps introduced in Definition~\ref{def::EtoRmaps}.

\begin{defn}
\label{defn::pdbfTopd}
Let $V$ be a PDF persistence module and $\pdbf(V)$ its persistence diagram. The \emph{undecorated persistence diagram} $\pd(V)$ of $V$ is a subset of $\overline{\R}^2 \times \N_{> 0}$  such that there exists a bijection $\X_V \colon \pdbf(V) \matching \pd(V)$ with the following property: if $\X_V([b,d,i]) = [b',d',i']$, then $b' = \pi( b)$ and $d' = \pi( d)$.
\end{defn}

To formulate an analog of Theorem~\ref{thm:genalgstability} for undecorated persistence diagrams we make use of the following functions. 

\begin{defn}
\label{defn:leftandrightinverse}
Let $\tau : \R \to \R$ be a monotone function. We define monotone functions $\tau_L, \tau^{\dagger}_L, \tau_R, \tau^{\dagger}_R : \overline{\R} \to \overline{\R}$ as:
\begin{enumerate}
\item $\tau_L(\pm\infty) = \lim_{x \to \pm\infty}\tau(x)$ and $\tau_L(x) := \lim_{y \to x^-} \tau(y)$ for $x \in \R$.  
\item $\tau_R(\pm\infty) = \lim_{x \to \pm\infty}\tau(x)$ and $\tau_R(x) := \lim_{y \to x^+} \tau(y)$ for $x \in \R$.
\item $\tau^{\dagger}_L(\pm\infty) := \pm\infty$ and $ \tau^{\dagger}_L(x) =  \inf\setof{ y \colon \tau(y) > x }$ for $x \in \R$.
\item $\tau^{\dagger}_R(\pm\infty) := \pm\infty$ and  $\tau^{\dagger}_R(x) :=  \inf\setof{ y \colon \tau(y) \geq x }$ for $x \in \R$.
\end{enumerate}
\end{defn}

\begin{prop}
\label{prop:leftandrightinverseconnections}
Let $\tau : \R \to \R$ be a monotone function. We have that $\tau_L = (\tau^{\downarrow})_-$, $\tau_R = (\tau^{\uparrow})_+$, $\tau^{\dagger}_L = (\tau^{\star})_+$, and $\tau^{\dagger}_R = (\tau^{\star})_-$.  Moreover, the pair $(\tau^{\dagger}_R, \tau_R)$ is a Galois connection and the pair $(\tau_L, \tau^{\dagger}_L)$ is a Galois connection.
\end{prop}

\begin{proof}
To establish the first part of the result, one performs a routine verification (which we omit) that the functions defined in Definition~\ref{defn:leftandrightinverse} could have been alternatively defined using the concepts in Definition~\ref{defn:promotions} and Definition~\ref{defn:demotions} according to the formulas given. 
Now the moreover part. Since $(\tau_L, \tau^{\dagger}_L) = ((\tau^{\downarrow})_-, (\tau^{\star})_+)$ and $(\tau^{\dagger}_R, \tau_R) = ((\tau^{\star})_-, (\tau^{\uparrow})_+)$, the result follows from Proposition~\ref{prop:endpointgaloisconnections} and   Proposition~\ref{prop:demotionconnection}.
\end{proof}

\begin{prop}
\label{prop:Matching}
Let $(V, \varphi_V)$ and $(W, \varphi_W)$ be persistence modules such that 
$(V,W)$ are $(\tau, \sigma)$-interleaved. Then there exists a matching $ \X \colon \pd(V) \matching \pd(W)$ with the following properties. If $\X( [b, d,i ]) = [ b', d', i' ]$, then
\begin{align*}
\tau^{\dagger}_R (b') &\leq b \leq (\sigma_R \circ \tau_R \circ \tau^{\dagger}_L) (b'),\\
\tau^{\dagger}_R (d') &\leq d \leq (\sigma_R \circ \tau_R \circ \tau^{\dagger}_L) (d'),\\
(\tau_L \circ \tau^{\dagger}_R \circ \sigma^{\dagger}_R)(b) &\leq b' \leq \tau_R (b), \\
(\tau_L \circ \tau^{\dagger}_R \circ \sigma^{\dagger}_R)(d) &\leq d' \leq \tau_R (d).
\end{align*}
Moreover, if $[b,d,i] \in \pd(V)$ is unmatched, then
\[
 d \leq (\sigma_R \circ \tau_R)(b),
\] 
and  if $[b',d',i'] \in \pd(W)$ is unmatched, then 
\[
d' \leq (\tau_R \circ \sigma_R \circ \tau_R \circ \tau^{\dagger}_L) (b').
\]
\end{prop}

\begin{proof}
We consider the matching $\X \colon \pd(V) \matching \pd(W)$ defined by 
\[
\X := \X_W \circ \X' \circ \X^{-1}_V, 
\]
where $\X_V, \X_W$ are bijections given by Definition~\ref{defn::pdbfTopd} and $\X'$ is the matching from Theorem~\ref{thm:genalgstability}. We start by proving the inequalities for the end points. We only need to show that  $\tau^{\dagger}_R (b')  \leq b$, $\tau^{\dagger}_R (d') \leq d$, $(\tau_L \circ \tau^{\dagger}_R \circ \sigma^{\dagger}_R)(b) \leq b'$, and $(\tau_L \circ \tau^{\dagger}_R \circ \sigma^{\dagger}_R)(d) \leq d'$ since by Proposition~\ref{prop:leftandrightinverseconnections} the other inequalities can be recovered using Galois connections (e.g. $\tau^{\dagger}_R (b') \leq b$ if and only if $b' \leq \tau_R(b)$). 

We only prove $\tau^{\dagger}_R (b')  \leq b$ since the rest  can be obtain by using similar arguments. Let $[ c, e, j] = \X_V^{-1}([b,d,i])$ and $[c', e', j'] = \X_W^{-1}([b',d',i'])$. Then $\X'([c, e, j]) = [ c', e', j']$. It follows from Corollary~\ref{cor:genalgstab_invertible} that $\tau^\star (c') \leq c$, and so $\pi \circ \tau^\star (c') \leq  \pi (c)$. Since $i^- \circ \pi(x) \leq \text{id}(x)$ for all $x \in \E$, we get $\pi \circ \tau^\star \circ i^- \circ \pi (c')  \leq  \pi (c)$. By Proposition~\ref{prop:leftandrightinverseconnections}, we get  $\tau^\dagger_R = \pi \circ \tau^\star \circ i^-$, which implies that $\tau^\dagger_R(\pi(c')) \leq \pi(c)$. It follows from the definition of $\X_V$ and $\X_W$ that $b'=\pi(c')$ and $b = \pi(c)$. Combining this with the previous inequality  yields $\tau^{\dagger}_R (b')  \leq b$.

Now we assume that  $[b,d,i] \in \pd(V)$ is  not in $\dom \X$. Again, let $[c,e,j] = \X^{-1}_V([b,d,i])$. By  Theorem~\ref{thm:genalgstability}, the point $[c, e, j]$ is $(\sigma \circ \tau)$-trivial, i.e. $e \leq \sigma^\uparrow \circ \tau^\uparrow (c)$.
Since $\text{id}(x) \leq i^+ \circ \pi(x)$ for all $x \in \E$,  we have $\pi(e) \leq \pi \circ \sigma^\uparrow \circ (i^+ \circ \pi) \circ \tau^\uparrow \circ (i^+ \circ \pi) (c)$.   Proposition~\ref{prop:leftandrightinverseconnections} implies that $\pi(e) \leq \sigma_R \circ \tau_R (\pi(c))$. By using $\pi(c) = b$ and $\pi(c') = b'$, we obtain that  $d \leq \sigma_R \circ \tau_R (b)$. The inequality for the points $[b',d',i'] \in \pd(W)$ that are  not in $\im \X$ can be achieved by using similar methods as above  and is left to the reader. 
\end{proof}

\begin{cor}
\label{cor:genalgstab_invertible_pi}
Let $(V, \varphi_V)$ and $(W, \varphi_W)$ be persistence modules such that $(V,W)$ are $(\tau, \sigma)$-interleaved. 
Suppose that  the maps $\tau$ and $\sigma$ are bijections. If $\X \colon \pd(V) \to \pd(W)$ is the matching given by Proposition~\ref{prop:Matching} and $\X([ b, d,i ]) = [ b', d', i']$, then
\[
\sigma^{-1}(b) \leq b' \leq {\tau}(b) \quad \text{and} \quad {\sigma}^{-1}(d) \leq d' \leq {\tau}(d).
\]
If $[b,d,i] \in \pd(V)$ is  unmatched, then  $d \leq (\sigma \circ \tau)(b),$ and  if $[b',d',i'] \in \pd(W)$ is  not in $\im \X$, then $d' \leq (\tau \circ \sigma) (b').$
\end{cor}

\begin{proof}
If $\tau$ and $\sigma$ are invertible, then  $\tau_L = \tau_R = \tau$, $\sigma_L = \sigma_R = \sigma$, $\tau^\dagger_L = \tau^\dagger_R = \tau^{-1}$, and $\sigma^\dagger_L = \sigma^\dagger_R = \sigma^{-1}$. The proof is obtained by  evaluating    expressions in Proposition~\ref{prop:Matching}.
 \end{proof}

\section{Applications}
\label{sec:examples}

We illustrate the use of results obtained in Section~\ref{sec::Stbility}  through a series of applications.
Our first example  examines the relationship between $\mathbb{Z}$-indexed and $\mathbb{R}$-indexed persistence modules.
The second example focuses on obtaining bounds on errors that arise from computational limitations to obtaining the true persistence diagrams for large point clouds.  
We conclude with a table indicating how to apply Theorem~\ref{thm:genalgstability}  for a variety of approximations that are commonly used.

\subsection{Discretizing a persistence module}
\label{sec:example_discretize}
The $\R$-indexed persistence module $V$ derived by considering the sublevel set filtration of a function $f\colon X\to \R$ provides a characterization of the topography of $f$.
However, in practice only a finite number of calculations can be performed.
A simple idealization is to assume that calculations are performed only at integer values of $f$.
This leads to the following definition.
\begin{defn}
The \emph{$\Z$-discretized  persistence module $V^\Z$} is defined as follows.
Set
\[
V^\Z_t := V_{\lfloor t \rfloor} \quad \text{and}\quad \varphi_{V^\Z}(s, t) := \varphi_{V}(\lfloor s \rfloor, \lfloor t \rfloor),
\]
where $\lfloor \cdot \rfloor$ is the floor function.
\end{defn}

The following proposition provides an answer to the following question: given the $\Z$-discretized  persistence module $V^\Z$, what are the constraints on the persistence diagram associated to the persistence module $V$?

\begin{prop}
\label{prop:discretized_pm}
If $V$ is an $\R$-indexed PFD persistence module and  $V^\Z$ is the associated  $\Z$-discretized PFD persistence module, then the following are true:
\begin{enumerate}
\item[(i)] $(V,V^\Z)$ are $(\tau, \sigma)$-interleaved, where $\tau(t) = t$ and $\sigma(t) = \lceil t \rceil$; and
\item[(ii)] there exists a matching $\X : \pd({V^\Z}) \matching \pd(V)$ such that if $\X ( [b, d, i] ) = [ b', d', i' ]$, then
\[
b - 1 \leq b' \leq b \quad\text{and}\quad d - 1 \leq d' \leq d.
\]
Additionally, any unmatched points $[ b, d, i ] \in \pd(V^\Z)$ satisfy $d = b + 1$, and unmatched points in $[ b', d', i' ] \in \pd(V)$ satisfy $d' \leq \lfloor b' + 1 \rfloor$.
\end{enumerate}
\end{prop}

\begin{proof}
(i) Define persistence module morphisms $\phi : V^\Z \rightarrow V(\tau)=V$ by $\phi_t := \varphi_V(\lfloor t \rfloor, t)$ and $\psi : V \rightarrow V^\Z(\sigma)$ by $\psi_t := \varphi_V(t, \lceil t \rceil)$.
Observe that
\begin{align*}
\psi(\tau)_t \circ \phi_t & = \psi_{\tau(t)} \circ \varphi_V(\lfloor t \rfloor,t) \\
& = \varphi_V(t,\lceil t \rceil)\circ \varphi_V(\lfloor t \rfloor,t) \\
& = \varphi_V(\lfloor t \rfloor, \lceil t \rceil) \\
& = \varphi_{V^\Z}(t,\sigma\circ\tau(t)).
\end{align*}
It is left to the reader to check that $\varphi(\sigma)_t\circ \psi_t = \varphi_V[t,(\tau\circ\sigma)(t)]$, and therefore that
$(V,V^\Z)$ are $(\tau, \sigma)$-interleaved.

(ii) Observe that if $[b,d,i]\in\pd({V^\Z})$, then $b,d\in\Z$. The proof now follows from Proposition~\ref{prop:Matching}  and the fact that for $t\in\R$
\[
\tau_R(t)=\tau_L(t) =\tau^\dagger_R(t)=\tau^\dagger_L(t)=t
\]
and
\[
\sigma^\dagger_R(t) = \lceil t - 1\rceil.
\]
\end{proof}

See Figure~\ref{fig:Zpd} for an illustration of an estimate of $\pd(V)$ from $\pd(V^\Z)$.

\subsection{Computing persistence diagrams for large point clouds}
\label{sec:example_point_clouds}

We now turn to the question of computing persistence diagrams for large point clouds.
For point clouds in arbitrary metric spaces, a standard approach makes use of a filtration of the associated Vietoris-Rips complex, which we define next.

\begin{defn}
\label{defn:vietoris-rips}
Let $(X,d)$ be a finite metric space with metric $d$.
The \emph{Vietoris-Rips complex of $X$ at scale $t$}, denoted by ${\cal R}(X, t)$, is the simplicial complex with vertices given by $X$ and containing the $N$-simplex $[ x_{i_0},\ldots, x_{i_N}]$ if and only if $d(x_{i_j},x_{i_k})\leq 2t$ for all $j,k = 0,\ldots, N$.

The collection $\{{\cal R}(X, t)\}_{t \in \R}$ is called the \emph{Vietoris-Rips filtration associated to $X$}.
\end{defn}

\begin{defn}
\label{defn:vietoris-rips-induced-pm}
Let $(X,d)$ be a finite metric space.
Fix a field $\sk$.
The \emph{persistence module induced by the Vietoris-Rips filtration associated to $X$}, denoted by $M^{\cal R}(X)$, is defined via simplicial homology as follows:
\[
M^{\cal R}(X)_t := H_*({\cal R}(X, t), \sk), \qquad t \in \R
\]
and the transition maps $\varphi_{M^{\cal R}(X)}(t,s)$ are the associated linear maps on homology induced by the inclusion maps $j_{X;t,s}\colon {\cal R}(X, t)\to {\cal R}(X, s)$.
\end{defn}

We remark that given a finite metric space $X$, the induced persistence module $M^{\cal R}(X)$ is a PFD persistence module.

Observe that for large $X$, the computational cost of determining  $H_*({\cal R}(X, t),\sk)$ grows rapidly as a function of $t$.
If  $Y\subset X$, then one expects that it is cheaper to compute $H_*({\cal R}(Y, t),\sk)$.
The goal of this section is twofold:
first, to quantify the difference between $M^{\cal R}(X)$ and $M^{\cal R}(Y)$; and second, to suggest an iterative procedure, motivated by \cite{dey_fan_wang}, for obtaining reasonable approximations of $M^{\cal R}(X)$.

\subsubsection{Subsampling a large point cloud}
\label{subsec:pointcoud_subsample}

\begin{defn}
Let $(X,d)$ be a finite metric space.
A subset $Y\subset X$ is a \emph{$\delta$-approximation} of $X$ if for every $x \in X$ there exists a $y \in Y$ such that $d(x, y) \leq \delta$.
\end{defn}

\begin{prop}
\label{prop:subsample_interleaving}
If $(X,d)$ is a finite metric space and $Y$ is a $\delta$-approximation of $X$, then the following approximations hold.
\begin{enumerate}
\item[(i)] The persistence modules $(M^{\cal R}(Y),M^{\cal R}(X))$ are $(\tau, \sigma)$-interleaved, where $\tau(t) = t$ and $\sigma(t) = t+\delta$.
\item[(ii)] There exists a matching $\X : \pd({M^{\cal R}(Y)}) \matching \pd({M^{\cal R}(X)})$ such that
if $\pd_\X( b , d, i ) = ( b' , d', i' )$, then $b - \delta \leq b' \leq b$ and $d - \delta \leq d' \leq d$.
Moreover, all unmatched points in $\pd({M^{\cal R}(Y)})$ and $\pd({M^{\cal R}(X)})$ are at most $\delta$ above the diagonal.
\end{enumerate}
\end{prop}

Figure~\ref{fig:subsamples}(left) provides an illustration of Proposition~\ref{prop:subsample_interleaving}(ii). Observe that since $\tau$ and $\sigma$ are invertible, Proposition~\ref{prop:subsample_interleaving}(ii) follows from Proposition~\ref{prop:subsample_interleaving}(i) and Corollary~\ref{cor:genalgstab_invertible_pi}.
The proof of Proposition~\ref{prop:subsample_interleaving}(i) occupies the remainder of this section.
We begin with some preliminary arguments.

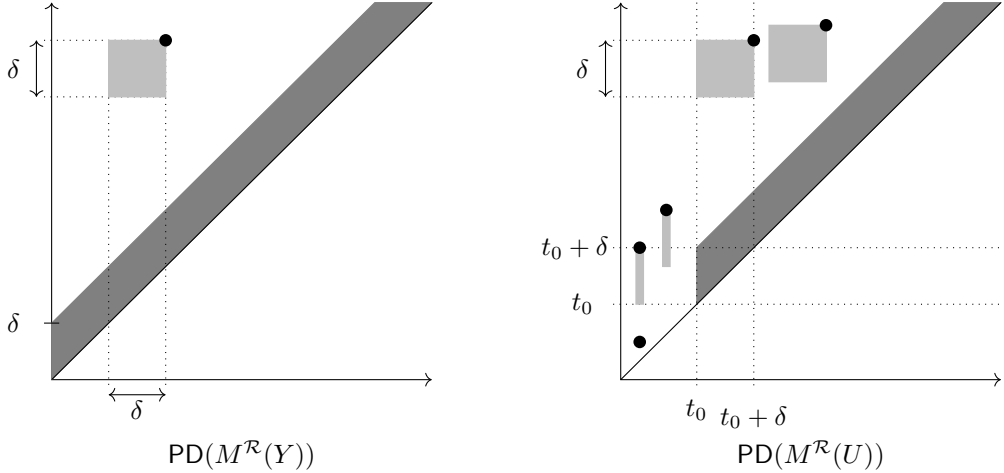
\begin{figure}[t]
\begin{picture}(300,200)

\put(0,0){
\begin{tikzpicture}
[scale=1]
		\draw [<->]  (5,0) --(0,0)  --(0,5);
		
		\fill[draw=black,color=gray]  (0,0) -- (0,0.75) -- (4.25,5) -- (5,5) -- (0,0);
		
		\draw [] (0,0) -- (5,5);
		
		 
		 \draw [dotted] (1.5,4.5) --(1.5,-0.1); 
		 \draw [dotted] (0.75,4.5) --(0.75,-0.1); 
		 
		 \draw [<->]  (0.75,-0.2) --(1.5,-0.2);
		 \draw(1.12,-0.4) node{$\delta$};
		 
		 \draw [dotted] (-0.1,4.5) --(1.5,4.5); 
		 \draw [dotted] (-0.1,3.75) --(1.5,3.75); 
		 
		 \draw [<->]  (-0.2,3.75) --(-0.2,4.5);
		 \draw(-0.5,4.12) node{$\delta$};

		 \fill[draw=black,color=lightgray]  (1.5,4.5) -- (1.5,3.75) -- (0.75,3.75) -- (0.75,4.5) -- (1.5,4.5);
		 
		 \node at (1.5,4.5) [circle,fill, scale=0.5] {};
		  
		 \draw [] (-0.1,0.75) --(0.1,0.75);
		 
		\draw(-0.5, 0.75) node{$\delta$};
		
		\draw(2.5,-1) node{$\pd({M^{\cal R}(Y)})$};   
\end{tikzpicture}
}		

\put(200,0){
\begin{tikzpicture}
[scale=1]
		\draw [<->]  (5,0) --(0,0)  --(0,5);
		
		\fill[draw=black,color=gray]  (1.0,1.0) -- (1.0,1.75) -- (4.25,5) -- (5,5) -- (1.0,1.0);
		
		\draw [] (0,0) -- (5,5);
		
		 
		 \draw [dotted] (1.0,5.0) --(1.0,-0.1); 
		 \draw [dotted] (1.75,5) --(1.75,-0.1); 
		 
		 
		 \draw [dotted] (-0.1,1.0) --(5,1.0); 
		 \draw [dotted] (-0.1,1.75) --(5,1.75);
		 \draw [dotted] (-0.1,4.5) --(1.75,4.5); 
		 \draw [dotted] (-0.1,3.75) --(1.75,3.75); 
		 
		 \draw [<->]  (-0.2,3.75) --(-0.2,4.5);
		 \draw(-0.5,4.12) node{$\delta$};

		 \fill[draw=black,color=lightgray]  (1.75,4.5) -- (1.75,3.75) -- (1,3.75) -- (1,4.5) -- (1.75,4.5);
		 \fill[draw=black,color=lightgray]  (2.7,4.7) -- (2.7,3.95) -- (1.95,3.95) -- (1.95,4.7) -- (2.7,4.7);
		 \fill[draw=black,color=lightgray]  (0.20,1.0) rectangle (0.30, 1.75);
		 \fill[draw=black,color=lightgray]  (0.55,1.5) rectangle (0.65, 2.25);
		 
		 \node at (0.25,0.5) [circle,fill, scale=0.5] {};
		 \node at (0.25,1.75) [circle,fill, scale=0.5] {};
		 \node at (0.6,2.25) [circle,fill, scale=0.5] {};
		 \node at (1.75,4.5) [circle,fill, scale=0.5] {};
		 \node at (2.7,4.7) [circle,fill, scale=0.5] {};

		 
		\draw(-0.6,1.75) node{$t_0+\delta$};
		
		\draw(-0.5,1.0) node{$t_0$};
		\draw(1.75,-0.5) node{$t_0+\delta$};
		\draw(1.0,-0.4) node{$t_0$};
		
		\draw(2.5,-1) node{$\pd({M^{\cal R}(U)})$};   		   
\end{tikzpicture}
}		

\end{picture}
\caption{ (left) Schematic diagram illustrating the quality of the matching from Proposition~\ref{prop:subsample_interleaving}. 
The persistence diagram $\pd({M^{\cal R}(Y)})$ is shown.
The dark gray region indicates the possible locations of the unmatched points for the persistence diagrams $\pd({M^{\cal R}(Y)})$ and $\pd({M^{\cal R}(X)})$.  
The light gray region indicates the possible location of the point of  $\pd({M^{\cal R}(X)})$ that is matched to the point of $\pd({M^{\cal R}(Y)})$ that is shown. 
(right) Schematic diagram illustrating the quality of the matching from Corollary~\ref{cor:subsample_interleaving}.
The persistence diagram $\pd({M^{\cal R}(U)})$ is shown.
Persistence points for the persistence diagrams $\pd({M^{\cal R}(U)})$ and $\pd({M^{\cal R}(X)})$ agree in the region $[0,t_0)\times [0,t_0)$.
Unmatched points for the persistence diagrams $\pd({M^{\cal R}(U)})$ and $\pd({M^{\cal R}(X)})$ will lie in the dark gray region. 
The light gray region indicates the possible location of the point of  $\pd({M^{\cal R}(X)})$ that is matched to the point of $\pd({M^{\cal R}(U)})$ that is shown.
}
\label{fig:subsamples}
\end{figure}

\begin{lem}
\label{lem:simplicialmap}
Let $Y,Y'\subseteq X$ and $\delta \geq 0$.
If $\gamma: Y \rightarrow Y'$ satisfies $d(x, \gamma(x)) \leq \delta$ for all $x \in Y$, then
$\tilde{\gamma}_t: \mathcal{R}(Y, t) \rightarrow \mathcal{R}(Y', t+\delta)$,
defined by 
\[
\tilde{\gamma}_t[x_0, \cdots, x_k] = [\gamma(x_0), \cdots, \gamma(x_k)]\quad \text{for any simplex $[x_0, \cdots, x_k]\in \mathcal{R}(Y, t)$},
\] 
is a simplicial map.
\end{lem}

\begin{proof}
To prove that $\tilde{\gamma}$ is a simplicial map, we need to show that for every $k$-simplex $[x_0, ... , x_k] \in \mathcal{R}(Y, t)$, the $k$-simplex $[\gamma(x_0), ..., \gamma(x_k)]$ is a simplex in $\mathcal{R}(Y', t + \delta)$.
Since the simplices in a Vietoris-Rips complex are fully determined by its $1$-skeleton, we only need to show that the $1$-skeleton of $\mathcal{R}(Y, t)$ is mapped to the $1$-skeleton of $\mathcal{R}(Y', t+\delta)$.
Recall that $[x,y]$ is an edge in $\mathcal{R}(Y, t)$ if and only if $d(x, y) \leq 2t$.
Thus, we have
\begin{align*}
d(\gamma(x), \gamma(y)) &\leq d(\gamma(x), x) + d(x, y) + d(y, \gamma(y)) \\
&\leq \delta + 2t + \delta \\
&= 2(t + \delta),
\end{align*}
and so $[\gamma(x), \gamma(y)]$ is either a $1$-simplex or a $0$-simplex in $\mathcal{R}(Y', t+ \delta)$.
\end{proof}

Let $Y\subset X$ be a $\delta$-approximation and let $\iota_t \colon {\cal R}(Y, t)\to  {\cal R}(X, t)$ denote the inclusion map.
Set 
\[
\phi_t := i_{t*}\colon H_*({\cal R}(Y, t))\to  H_*({\cal R}(X, t)).
\]
Since $Y$ is a $\delta$-approximation, there exists $\gamma\colon X\to Y$ such that $d(x, \gamma(x)) \leq \delta$ for all $x \in X$, and $\gamma(y)=y$ for all $y\in Y$. 
By Lemma~\ref{lem:simplicialmap}, $\tilde{\gamma}_t: \mathcal{R}(X, t) \rightarrow \mathcal{R}(Y, t+\delta)$ is a simplicial map and hence we can define
\[
\psi_t := \tilde{\gamma}_{t*} \colon H_*(\mathcal{R}(X, t)) \rightarrow H_*(\mathcal{R}(Y, t+\delta)).
\]
Our goal is to show that $\phi\colon M^{\cal R}(Y)\to M^{\cal R}(X)$ and $\psi \colon M^{\cal R}(X) \to M^{\cal R}(Y)(\delta)$ are persistence module morphisms that guarantee that the persistence modules $(M^{\cal R}(Y),M^{\cal R}(X))$ are $(\tau, \sigma)$-interleaved, and therefore, provide a proof of Proposition~\ref{prop:subsample_interleaving}(i).

Observe that $\gamma_t\circ \iota_t = j_{Y;t,t+\delta}$ and hence
\begin{equation}
\label{eq:psiphi}
\psi(\tau)_t\circ \phi_t = \psi_t \circ \phi_t = \varphi_{M^{\cal R}(Y)}[t,t+\delta] =  \varphi_{M^{\cal R}(Y)}[t,(\sigma\circ \tau)t]. 
\end{equation}
The challenge is to show that the middle equality holds for
\begin{equation}
\label{eq:phipsi}
 \phi(\sigma)_t\circ \psi_t = \phi_{t+\delta}\circ \psi_t  = \varphi_{M^{\cal R}(X)}[t,t+\delta] = \varphi_{M^{\cal R}(X)}[t,( \tau\circ\sigma)t].
\end{equation}
For purposes of the next section, we prove a more general result than necessary.

\begin{lem}
\label{lem:simplicialhomotopy}
Consider $Y,Y'\subseteq X$ and  $\delta \geq 0$.
Let $\iota'_t \colon \mathcal{R}(Y', t) \to \mathcal{R}(X, t + \delta)$ and $\iota_t \colon \mathcal{R}(Y, t) \to \mathcal{R}(X, t )$ be the simplicial maps induced by inclusion.
If $\gamma \colon Y' \to Y$ satisfies $d(y, \gamma(y)) \leq \delta$ for all $y \in Y'$ and  $\tilde{\gamma}_t: \mathcal{R}(Y', t) \to \mathcal{R}(Y, t+\delta)$ is the simplicial map as defined in Lemma~\ref{lem:simplicialmap}, then 
$\iota_{(t+\delta)} \circ \tilde{\gamma}_t$ and  $\iota'_t$ are homotopic and hence
\[
\iota_{(t+\delta)*} \circ \tilde{\gamma}_{t*} = \iota'_{t*}.
\]
\end{lem}

\begin{proof}
To prove this we make use of the theory of simplicial sets \cite{weibel, friedman2012} and begin with the remark that by \cite[Lemma 8.3.13, Theorem 8.3.8]{weibel} it is sufficient to prove that $\iota_{(t+\delta)} \circ \tilde{\gamma}_t$ and  $\iota'_t$ are homotopic.

Given a simplicial complex $\cK$ let $\bar{\cK}$ denote the associated simplicial set.  To establish notation let $\bar{\cK}_k$ denote the $k$-dimensional simplices in $\cK$ and let $d_i\colon \bar{\cK}_k \to \bar{\cK}_{k-1}$ and
$s_i\colon \bar{\cK}_k \to \bar{\cK}_{k+1}$ be the delete and duplicate $i$-th vertex operations defined by
\[
d_i[v_0,\ldots, v_k] := [v_0,\ldots,v_{i-1}, v_{i+1},\ldots, v_k]
\]
and
\[
s_i[v_0,\ldots, v_k] := [v_0,\ldots,v_{i}, v_{i},\ldots, v_k].
\]

We claim that the functions $h_i\colon \bar{\mathcal{R}}(Y', t)_k \to \bar{\mathcal{R}}(X, t + \delta)_{k+1}$, $i=0,\ldots,k$, defined by 
\[
h_i([x_0,\ldots, x_k]) = [x_0,\ldots,x_i,\gamma(x_i),\ldots, \gamma(x_k)]
\]
provide a simplicial homotopy between $\iota_{(t+\delta)} \circ \tilde{\gamma}_t$ and  $\iota'_t$.
Recall that to justify this claim, it is sufficient to verify the following equalities:
\[
d_0 h_0  =  \iota_{(t+\delta)} \circ \tilde{\gamma}_t \quad\text{and}\quad
d_{k+1} h_k = \iota'_t
\]
\[
d_i h_j = \begin{cases}
h_{j-1}d_i & \text{if $i<j$} \\
d_i h_{i-1} & \text{if $i=j\neq 0$} \\
h_jd_{i-1} & \text{if $i > j+1$}\\
\end{cases}
\quad\text{and}\quad
s_ih_j = \begin{cases}
h_{j+1}s_i & \text{if $i\leq j$} \\
h_j s_{i-1} & \text{if $i>j$.}\\
\end{cases}
\]
We leave these calculations to the reader.
\end{proof}

\begin{proof}[Proof of Proposition~\ref{prop:subsample_interleaving}(i)]
As indicated above, the proof of Proposition~\ref{prop:subsample_interleaving}(i) follows from \eqref{eq:psiphi}, which has already been justified, and \eqref{eq:phipsi}, which follows from Lemma~\ref{lem:simplicialhomotopy} under the assumption that $Y'=X$.
\end{proof}

\subsubsection{Stitching persistence modules}
\label{subsec:merge_module}

Proposition~\ref{prop:subsample_interleaving} demonstrates that  given a finite metric space $(X,d)$, bounds on the persistence diagram $\pd(M^\cR(X))$ can be determined from $\pd(M^\cR(Y))$ under the assumption that $Y$ is a $\delta$-approximation of $X$.
The motivation for using $Y$ is that for large $t$ it may not be feasible to compute $H_*(\cR(X,t);\sk)$.
However, for small $t$, the size of the complex $\cR(X,t)$ is on the order of the size of $X$.
Furthermore, if $t<\delta$, then $\cR(Y,t)$ will fail to capture the fine geometric structure of $X$, and hence, as indicated by Figure~\ref{fig:subsamples}(left), points in this region of $M^\cR(X)$ will be unmatched with respect to $M^\cR(Y)$.
For this reason, we would like to construct a persistence module determined by $M^\cR(X)$ in the range  $t<\delta$ and $M^\cR(Y)$ in the range $t \geq \delta$.
This will provide finer information, as illustrated in Figure~\ref{fig:subsamples}(right). 
Observe that this suggests the need to be able to stitch together persistence modules and motivates the following definition.

\begin{defn}
\label{defn:PMstitching}
Let $(V, \varphi_V), (V', \varphi_{V'})$ and $(W, \varphi_W)$ be persistence modules such that 
$(V,W)$ are $(\tau, \sigma)$-interleaved via the morphisms $\phi : V \rightarrow W(\tau)$ and ${\psi : W \rightarrow V(\sigma)}$, and that $(V',W)$ are $(\tau', \sigma')$-interleaved via ${\phi' : V' \rightarrow W(\tau')}$ and ${\psi' : W \rightarrow V'(\sigma')}$.
A point $t_0\in\R$ is a \emph{$(V,V')$ stitch point} if
\[
t_0 \leq (\sigma'\circ \tau)(t_0).
\]
The \emph{$(V,V')$ persistence module stitched through W at stitch point $t_0$} consists of vector spaces
\[
U_t = \begin{cases}
V_t 		& \text{if $t<t_0$} \\
V_{t_0} 	& \text{if $t_0\leq t < (\sigma' \circ \tau)(t_0)$} \\
V'_t		& \text{if $(\sigma' \circ \tau)(t_0) \leq t$}
\end{cases}
\]
and linear maps
\[
\varphi_{U}(s, t) = \begin{cases}
\varphi_{V}(\min\{s, t_0\}, \min\{t, t_0\})		& \text{if $s\leq t < (\sigma' \circ \tau)(t_0)$} \\
\varphi_{V'}(t,(\sigma' \circ \tau)(t_0)) \circ \psi'_{\tau(t_0)} \circ \phi_{t_0} \circ \varphi_V(\min\{s, t_0\}, t_0) 	& \text{if $s < (\sigma' \circ \tau)(t_0) < t$} \\
\varphi_{V'}(s, t)		& \text{if $(\sigma' \circ \tau)(t_0) \leq t$}.
\end{cases}
\]
We denote this persistence module by $U = U(W;V,t_0,V')$ and for the sake of simplicity refer to it as the \emph{stitched persistence module}.
\end{defn}

The following diagram (with unlabeled arrows assumed to be the appropriate transition maps) shows the idea behind the vector spaces of $U$ and the transition maps $\varphi_U(s, t)$ for $s < (\sigma' \circ \tau)(t_0) \leq t$.
\begin{center}
\begin{displaymath}
\setlength\mathsurround{0pt}
\begin{tikzcd}
V_s \arrow[rr, rightarrow] && V_{t_0} \arrow[rr, dashrightarrow, "\psi'_{\tau(t_0)} \circ \phi_{t_0}"] \arrow[dr, rightarrow, "\phi_{t_0}"] && V'_{(\sigma' \circ \tau)(t_0)} \arrow[rr,rightarrow] && V'_t \\
W_s \arrow[rrr, rightarrow] &&& W_{\tau(t_0)} \arrow[rrr,rightarrow] \arrow[ur, rightarrow, "\psi'_{\tau(t_0)}"] &&& W_t
\end{tikzcd}
\end{displaymath}
\end{center}

\begin{prop}
\label{prop:stitching_lemma}
Let $(V, \varphi_V), (V', \varphi_{V'}),$ and $(W, \varphi_W)$ be PFD persistence modules such that 
$(V,W)$ are $(\tau, \sigma)$-interleaved via the morphisms $\phi : V \rightarrow W(\tau)$ and ${\psi : W \rightarrow V(\sigma)}$, and that $(V',W)$ are $(\tau', \sigma')$-interleaved via ${\phi' : V' \rightarrow W(\tau')}$ and ${\psi' : W \rightarrow V'(\sigma')}$. 
Assume that $t_0$ is a stitch point.
If $U = U(W;V,t_0,V')$ is the stitched persistence module, then the following statements are true:
\begin{enumerate}
\item[(i)] $U$ is a PFD persistence module;
\item[(ii)] $(U,W)$ are $(\eta, \rho)$-interleaved where 
 \begin{align*}
   \eta(t) &= \left\{
     \begin{array}{ll}
       \tau(t) & \text{if\;\;\;} t \leq t_0 \\
       \tau(t_0) & \text{if\;\;\;} t_0 < t < (\sigma' \circ \tau)(t_0) \\
       \tau'(t) & \text{if\;\;\;} (\sigma' \circ \tau)(t_0) \leq t;
     \end{array}
   \right.\\
   \rho(t) &= \left\{
     \begin{array}{ll}
       \sigma(t) & \text{if\;\;\;} t < \sigma^{-1}(t_0) \\
       (\sigma' \circ \tau)(t_0) & \text{if\;\;\;} \sigma^{-1}(t_0) \leq t < \tau(t_0)  \\
       \sigma'(t) & \text{if\;\;\;} \tau(t_0) \leq t.
     \end{array}
   \right.
\end{align*}
\end{enumerate}
\end{prop}

\begin{proof}  (i) Since $V$ and $V'$ are PFD persistence modules, $U_t$ is finite dimensional for each $t\in \R$.
It is left to the reader to check that $\varphi_U(t,t) = \text{id}_{U_t}$ and $\varphi_U(s, t) \circ \varphi_U(r, s) = \varphi_U(r, t)$ for every $r \leq s \leq t$ in $\R$. 

(ii) To show that $(U,W)$ are $(\eta, \rho)$-interleaved we will show that the morphisms $\bar{\phi}: U \rightarrow W(\eta)$ and $\bar{\psi} : W \rightarrow U(\rho)$, where 
 \begin{align*}
   \bar{\phi}_t &= \left\{
     \begin{array}{ll}
       \phi_t & \text{if\;\;\;} t \leq t_0 \\
       \phi_{t_0} & \text{if\;\;\;} t_0 < t < (\sigma' \circ \tau)(t_0) \\
       \phi'_t & \text{if\;\;\;} (\sigma' \circ \tau)(t_0) \leq t, 
     \end{array}
   \right.\\
   \bar{\psi}_t &= \left\{
     \begin{array}{ll}
       \psi_t & \text{if\;\;\;} t < \sigma^{-1}(t_0) \\
       \psi'_{\tau(t_0)} \circ \varphi_W(t, \tau(t_0)) & \text{if\;\;\;} \sigma^{-1}(t_0) \leq t < \tau(t_0)  \\
       \psi'_t & \text{if\;\;\;} \tau(t_0) \leq t,
     \end{array}
   \right.
\end{align*} 
give the desired interleaving of $U$ and $W$.
To show that $\bar{\phi}$ and $\bar{\psi}$ are persistence module morphisms, we first note that the monotone functions $\eta$ and $\rho$ line up with the indices of the shifts of $\bar{\phi}$ and $\bar{\psi}$ by inspection.
For what follows, let $s \leq t \in \R$.

We will now show that $\bar{\phi} \colon U \rightarrow W(\eta)$ is a persistence module morphism.
If $s \leq t < (\sigma' \circ \tau)(t_0)$, then $\bar{\phi}_s = \phi_{\min\{s, t_0\}}$ and $\bar{\phi}_t = \phi_{\min\{t, t_0\}}$, and if $(\sigma' \circ \tau)(t_0) \leq s \leq t$, then $\bar{\phi}_s = \phi'_s$ and $\bar{\phi}_t = \phi'_t$, and so these cases hold.
Now suppose that $s < (\sigma' \circ \tau)(t_0) \leq t$.
The diagram
\begin{center}
\begin{displaymath}
\setlength\mathsurround{0pt}
\begin{tikzcd}
V_s \arrow[dr, rightarrow, "\phi_{s}"] \arrow[r, rightarrow] & V_{t_0} \arrow[rr, dashrightarrow, "\psi'_{\tau(t_0)} \circ \phi_{t_0}"] \arrow[dr, rightarrow, "\phi_{t_0}"] && V'_{(\sigma' \circ \tau)(t_0)} \arrow[dr,rightarrow,"\phi'_{(\sigma' \circ \tau)(t_0)}"] \arrow[r,rightarrow] & V'_t \arrow[dr,rightarrow,"\phi'_{t}"] \\
& W_{\tau(s)} \arrow[r, rightarrow] & W_{\tau(t_0)} \arrow[rr,rightarrow] \arrow[ur, rightarrow, "\psi'_{\tau(t_0)}"] && W_{(\tau' \circ \sigma' \circ \tau)(t_0)} \arrow[r,rightarrow] & W_{\tau'(t)}
\end{tikzcd}
\end{displaymath}
\end{center}
where unlabeled arrows are transition maps, commutes since both $\phi$ and $\phi'$ are persistence module morphisms and  $(V',W)$ are $(\tau', \sigma')$-interleaved.
It follows that $\bar{\phi}_t \circ \varphi_U(s, t) = \varphi_{W(\eta)}(s, t) \circ \bar{\phi}_s$, and thus $\bar{\phi}\colon U \rightarrow W(\eta)$ is a persistence module morphism.

Now we will show that $\bar{\psi} : W \rightarrow U(\rho)$ is a persistence module morphism.
If $s \leq t < \sigma^{-1}(t_0)$, then $\bar{\psi}_s = \psi_s$ and $\bar{\psi}_t = \psi_t$, and if $\tau(t_0) \leq s \leq t$, then $\bar{\psi}_s = \psi'_s$ and $\bar{\psi}_t = \psi'_t$, so these cases are clear.
Suppose that $s < \sigma^{-1}(t_0) \leq t \leq \tau(t_0)$.
This choice of $s$ and $t$ yield $\bar{\psi}_s = \psi_s$ and $\bar{\psi}_t = \psi'_{\tau(t_0)} \circ \varphi_W(t, \tau(t_0))$ by definition.
Since $s < \sigma^{-1}(t_0)$, it follows that $s \leq (\tau \circ \sigma)(s) \leq \tau(t_0)$, and so the following diagram 
\begin{center}
\begin{displaymath}
\setlength\mathsurround{0pt}
\begin{tikzcd}
& V_{\sigma(s)} \arrow[dr, rightarrow, "\phi_{s}"] \arrow[rr, rightarrow] & &  V_{t_0} \arrow[rr, dashrightarrow, "\psi'_{\tau(t_0)} \circ \phi_{t_0}"] \arrow[dr, rightarrow, "\phi_{t_0}"] && V'_{(\sigma' \circ \tau)(t_0)} \\
W_s \arrow[rr,rightarrow] \arrow[drr, rightarrow] \arrow[ur, rightarrow,"\psi_s"] & & W_{(\tau \circ \sigma)(s)} \arrow[rr, rightarrow] & & W_{\tau(t_0)} \arrow[ur, rightarrow, "\psi'_{\tau(t_0)}"] \\
& & W_t \arrow[urr, rightarrow] \arrow[uurrr, bend right=35, "\bar{\psi}_t"']
\end{tikzcd}
\end{displaymath}
\end{center}
where unlabeled arrows are transition maps, commutes. Hence, $\bar{\psi}_t \circ \varphi_W(s,t) = \varphi_{U(\rho)}(s, t) \circ \bar{\psi}_s$ in this case as well.
For $\sigma^{-1}(t_0) \leq s \leq t \leq \tau(t_0)$, then 
\begin{align*}
\bar{\psi}_t \circ \varphi_W(s,t) &= \psi'_{\tau(t_0)} \circ \varphi_W(t, \tau(t_0)) \circ \varphi_W(s,t) \\
&= \varphi_U((\sigma' \circ \tau)(t_0),(\sigma' \circ \tau)(t_0)) \circ \psi'_{\tau(t_0)} \circ \varphi_W(s, \tau(t_0))\\
&= \varphi_{U(\rho)}(s,t) \circ \bar{\psi}_s,
\end{align*}
and so this case also holds.
If $s < \sigma^{-1}(t_0) \leq \tau(t_0) < t$, the commutativity of the diagram
\begin{center}
\begin{displaymath}
\setlength\mathsurround{0pt}
\begin{tikzcd}
& V_{\sigma(s)} \arrow[dr, rightarrow, "\phi_{s}"] \arrow[r, rightarrow] &  V_{t_0} \arrow[rr, dashrightarrow, "\psi'_{\tau(t_0)} \circ \phi_{t_0}"] \arrow[dr, rightarrow, "\phi_{t_0}"] && V'_{(\sigma' \circ \tau)(t_0)} \arrow[r,rightarrow] & V'_{\sigma'(t)} \\
W_s \arrow[rr,rightarrow]  \arrow[ur, rightarrow,"\psi_s"] & & W_{(\tau \circ \sigma)(s)} \arrow[r, rightarrow] & W_{\tau(t_0)} \arrow[ur, rightarrow, "\psi'_{\tau(t_0)}"] \arrow[r,rightarrow] & W_t  \arrow[ur, rightarrow, "\psi'_t"]\\
\end{tikzcd}
\end{displaymath}
\end{center}
ensures that $\bar{\psi}_t \circ \varphi_W(s,t) = \varphi_{U(\rho)}(s, t) \circ \bar{\psi}_s$.
Finally, it remains to check the case when $\sigma^{-1}(t_0) \leq s \leq \tau(t_0) < t$.
This follows from extending the case when $\sigma^{-1}(t_0) \leq s \leq t \leq \tau(t_0)$ by the right-most commutative square in the previous diagram.
Thus, we have shown that $\bar{\phi} \colon U \rightarrow W(\eta)$ and $\bar{\psi}\colon U \rightarrow W(\eta)$ are persistence module morphisms.

We must now show that $\bar{\phi}$ and $\bar{\psi}$ satisfy the commutativity constraints of Definition~\ref{defn:general_interleaving}.
That is, for every $t \in \R$, we must show that 
\begin{equation}
\label{eq:stitching_1}
\bar{\psi}(\eta)_t \circ \bar{\phi}_t = \varphi_U(t, (\rho \circ \eta)(t))
\end{equation}
and
\begin{equation}
\label{eq:stitching_2}
\bar{\phi}(\rho)_t \circ \bar{\psi}_t = \varphi_W(t, (\rho \circ \eta)(t))
\end{equation}
for all $t \in \R$.
First we prove (\ref{eq:stitching_1}) for all $t \in \R$.
When $\tau(t) < \sigma^{-1}(t_0)$, then (\ref{eq:stitching_1}) follows from the fact that $V$ and $W$ are $(\tau, \sigma)$-interleaved.
Similarly, when $(\sigma' \circ \tau)(t_0) \leq t$, (\ref{eq:stitching_1}) holds since $V'$ and $W$ are $(\tau', \sigma')$-interleaved.
If $t_0 \leq t < (\sigma' \circ \tau)(t_0)$, then (\ref{eq:stitching_1}) holds by the definition of $(U, \varphi_U)$.
Now suppose that $t \leq t_0$ such that $\sigma^{-1}(t_0) \leq \tau(t)$.
Then (\ref{eq:stitching_1}) holds from the commutativity of the following diagram.
\begin{center}
\begin{displaymath}
\setlength\mathsurround{0pt}
\begin{tikzcd}
V_t \arrow[dr, rightarrow, "\phi_t"] \arrow[rr, rightarrow] & & V_{t_0} \arrow[dr, rightarrow, "\phi_{t_0}"] \arrow[rr, dashrightarrow, "\psi'_{\tau(t_0)} \circ \phi_{t_0}"]  && V'_{(\sigma' \circ \tau)(t_0)}\\
& W_{\tau(t)}\arrow[urrr, bend right=50, "\bar{\psi}_{\tau(t)}"'] \arrow[rr, rightarrow] & & W_{\tau(t_0)} \arrow[ur, rightarrow, "\psi'_{\tau(t_0)}"]
\end{tikzcd}
\end{displaymath}
\end{center}
Thus we have shown that (\ref{eq:stitching_1}) holds for every $t \in \R$.

We now verify that (\ref{eq:stitching_2}) holds for every $t \in \R$.
When $t < \sigma^{-1}(t_0)$, then (\ref{eq:stitching_2}) follows since $V$ and $W$ are $(\tau, \sigma)$-interleaved.
Similarly, if $\tau(t_0) \leq t$, then (\ref{eq:stitching_2}) follows since $V'$ and $W$ are $(\tau', \sigma')$-interleaved.
For $\sigma^{-1}(t_0) \leq t < \tau(t_0)$, then (\ref{eq:stitching_2}) follows by the commutativity of the following diagram.
\begin{center}
\begin{displaymath}
\setlength\mathsurround{0pt}
\begin{tikzcd}
& & V'_{(\sigma' \circ \tau)(t_0)} \arrow[dr, rightarrow, "\bar{\phi}_{(\sigma' \circ \tau)(t_0)}"] \\
W_t \arrow[urr, rightarrow, "\bar{\psi}_t"] \arrow[r, rightarrow] & W_{\tau(t_0)} \arrow[ur, rightarrow, "\psi'_{\tau(t_0)}"'] \arrow[rr, rightarrow] & & W_{(\tau' \circ \sigma' \circ \tau)(t_0)}
\end{tikzcd}
\end{displaymath}
\end{center}
Thus, we have also shown that (\ref{eq:stitching_2}) holds for every $t \in \R$.
It follows that $(\eta, \rho)$ is a translation pair.
\end{proof}

We conclude this section with a simple, but hopefully illustrative, application of Proposition~\ref{prop:stitching_lemma}.

\begin{figure}
\begin{picture}(300,220)

\put(0,110){
\begin{tikzpicture}
[scale=1]
		\draw [<->]  (2.5,0) --(0,0)  --(0,2.5);
		\draw(1.25,2.5) node{$\eta$};
		
		 \draw [] (0.75,0.1) --(0.75,-0.1);
		  \draw [] (1.5,0.1) --(1.5,-0.1);
		  
		\draw(0.75,-0.5) node{$t_0$};
		\draw(1.5,-0.5) node{$t_0+\delta$};
		  
		 \draw [dotted] (0.75,0.1) --(0.75,2.5);
		 \draw [dotted] (1.5,0.1) --(1.5,2.5); 
		  
		 \draw [] (-0.1,0.75) --(0.1,0.75);
		 \draw [] (-0.1,1.5) --(0.1,1.5);
		 
		\draw(-0.5, 0.75) node{$t_0$};
		\draw(-0.75,1.5) node{$t_0+\delta$};
		 
		 \draw [dotted] (0.1,0.75) --(2.5,0.75);
		 \draw [dotted] (0.1,1.5) --(2.5,1.5); 
		 
		 \draw [thick] (0.0,0.0) -- (0.75,0.75) -- (1.5,0.75);
		 \draw (1.5,0.75)  circle (0.08cm);
		  \draw [thick] (1.5,1.5) -- (2.5,2.5);
		  \node at (1.5,1.5) [circle,fill, scale=0.5] {};
		   
\end{tikzpicture}
}		

\put(150,110){
\begin{tikzpicture}
[scale=1]
		\draw [<->]  (2.5,0) --(0,0)  --(0,2.5);
		\draw(1.25,2.5) node{$\eta^\dagger_L$};
		
		 \draw [] (0.75,0.1) --(0.75,-0.1);
		  \draw [] (1.5,0.1) --(1.5,-0.1);
		  
		\draw(0.75,-0.5) node{$t_0$};
		\draw(1.5,-0.5) node{$t_0+\delta$};
		  
		 \draw [dotted] (0.75,0.1) --(0.75,2.5);
		 \draw [dotted] (1.5,0.1) --(1.5,2.5); 
		  
		 \draw [] (-0.1,0.75) --(0.1,0.75);
		 \draw [] (-0.1,1.5) --(0.1,1.5);
		 
		 
		 \draw [dotted] (0.1,0.75) --(2.5,0.75);
		 \draw [dotted] (0.1,1.5) --(2.5,1.5); 
		 
		 \draw [thick] (0.0,0.0) -- (0.75,0.75);
		 \draw (0.75,0.75)  circle (0.08cm);
		  \draw [thick]  (0.75,1.5) --(1.5,1.5) -- (2.5,2.5);
		  \node at (0.75,1.5) [circle,fill, scale=0.5] {};
		   
\end{tikzpicture}
}

\put(275,110){
\begin{tikzpicture}
[scale=1]
		\draw [<->]  (2.5,0) --(0,0)  --(0,2.5);
		\draw(1.25,2.5) node{$\eta^\dagger_R$};
		
		 \draw [] (0.75,0.1) --(0.75,-0.1);
		  \draw [] (1.5,0.1) --(1.5,-0.1);
		  
		\draw(0.75,-0.5) node{$t_0$};
		\draw(1.5,-0.5) node{$t_0+\delta$};
		  
		 \draw [dotted] (0.75,0.1) --(0.75,2.5);
		 \draw [dotted] (1.5,0.1) --(1.5,2.5); 
		  
		 \draw [] (-0.1,0.75) --(0.1,0.75);
		 \draw [] (-0.1,1.5) --(0.1,1.5);
		 
		 
		 \draw [dotted] (0.1,0.75) --(2.5,0.75);
		 \draw [dotted] (0.1,1.5) --(2.5,1.5); 
		 
		 \draw [thick] (0.0,0.0) -- (0.75,0.75);
		 \node at (0.75,0.75) [circle,fill, scale=0.5] {};
		  \draw [thick]  (0.75,1.5) --(1.5,1.5) -- (2.5,2.5);
		  \draw (0.75,1.5)  circle (0.08cm);
		 
\end{tikzpicture}
}	

\put(0,0){
\begin{tikzpicture}
[scale=1]
		\draw [<->]  (2.5,0) --(0,0)  --(0,2.5);
		\draw(1.25,2.5) node{$\rho$};
		
		 \draw [] (0.75,0.1) --(0.75,-0.1);
		  \draw [] (1.5,0.1) --(1.5,-0.1);
		  
		\draw(0.75,-0.5) node{$t_0$};
		\draw(1.5,-0.5) node{$t_0+\delta$};
		  
		 \draw [dotted] (0.75,0.1) --(0.75,2.5);
		 \draw [dotted] (1.5,0.1) --(1.5,2.5); 
		  
		 \draw [] (-0.1,0.75) --(0.1,0.75);
		 \draw [] (-0.1,1.5) --(0.1,1.5);
		 
		\draw(-0.5, 0.75) node{$t_0$};
		\draw(-0.75,1.5) node{$t_0+\delta$};
		 
		 \draw [dotted] (0.1,0.75) --(2.5,0.75);
		 \draw [dotted] (0.1,1.5) --(2.5,1.5); 
		 
		 \draw [thick] (0.0,0.0) -- (0.75,0.75);
		 \draw (0.75,0.75)  circle (0.08cm);
		 \draw [thick] (0.75,1.5) -- (1.75,2.5);
		  \node at (0.75,1.5) [circle,fill, scale=0.5] {};
		   
\end{tikzpicture}
}		

\put(150,0){
\begin{tikzpicture}
[scale=1]
		\draw [<->]  (2.5,0) --(0,0)  --(0,2.5);
		\draw(1.25,2.5) node{$\rho^\dagger_L =\rho^\dagger_R$};
		
		 \draw [] (0.75,0.1) --(0.75,-0.1);
		  \draw [] (1.5,0.1) --(1.5,-0.1);
		  
		\draw(0.75,-0.5) node{$t_0$};
		\draw(1.5,-0.5) node{$t_0+\delta$};
		  
		 \draw [dotted] (0.75,0.1) --(0.75,2.5);
		 \draw [dotted] (1.5,0.1) --(1.5,2.5); 
		  
		 \draw [] (-0.1,0.75) --(0.1,0.75);
		 \draw [] (-0.1,1.5) --(0.1,1.5);
		 
		 
		 \draw [dotted] (0.1,0.75) --(2.5,0.75);
		 \draw [dotted] (0.1,1.5) --(2.5,1.5); 
		 
		 \draw [thick] (0.0,0.0) -- (0.75,0.75) -- (1.5,0.75) -- (2.5,1.75);
		   
\end{tikzpicture}
}	

\end{picture}
\caption{ Functions $\eta$ and $\rho$ that provide an interleaving between the stitched persistence module $U = (M^\cR(X); M^\cR(X),t_0,M^\cR(Y))$ and $(M^\cR(X)$ where $Y$ is a $\delta$-approximation of a finite metric space $X$.}
\label{fig:etarho}
\end{figure}
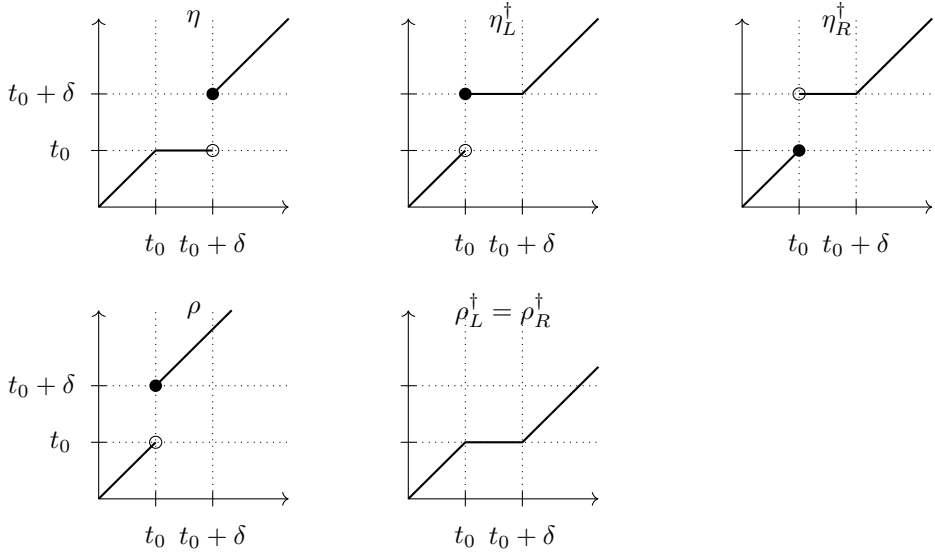

\begin{cor}
\label{cor:subsample_interleaving}
Let $(X,d)$ be a finite metric space and let $Y\subset X$ be a $\delta$-approximation of $X$ with $\delta > 0$.
Let $U = U(M^\cR(X); M^\cR(X),t_0,M^\cR(Y))$ be the stitched persistence module and  $\pd(U)$ its persistence diagram. If  $[ b , d, i ] \in \pd(U)$, then neither $b$ or $d$ is in the interval $(t_0,t_0 +\delta)$.
Moreover, there exists a matching $\X : \pd(U) \matching \pd({M^{\cal R}(X)})$ such that
if $\X( [ b , d, i ]) = [ b' , d', i' ]$ and 
\[
\begin{array}{lccccc}
\text{if} & b < d \leq t_0 & \text{then} & b' = b & \text{and} & d' = d ;\\
\text{if} & b \leq t_0 < t_0+\delta \leq d & \text{then} &   b' = b  &\text{and} & \max(t_0, d -\delta) \leq d' \leq d ;\\
\text{if} & b = t_0+\delta < d & \text{then} &   t_0 \leq b' \leq t_0+\delta  &\text{and} & \max(t_0, d -\delta) \leq d' \leq d  ;\\
\text{if} & t_0+\delta < b < d & \text{then} &   b-\delta \leq b' \leq b &\text{and}  &d-\delta \leq d' \leq d .
\end{array}
\]
All unmatched points $[b,d,i] \in\pd(U)$ satisfy 
\[
\begin{array}{cccr}
t_0 + \delta \leq b & \text{and} & d \leq b + \delta ,\\
\end{array}
\]
and  all unmatched points $[b',d',i']\in\pd({M^{\cal R}(X)})$ satisfy 
\[
\begin{array}{cccr}
t_0 < b' < d' \leq b' + \delta. \\
\end{array}
\]
\end{cor}

\begin{proof} 
It follows form the definition of the Vietoris-Rips filtration that every interval in $\pdbf(M^\cR(X))$ and $\pdbf(M^\cR(Y))$ has a closed left-hand endpoint and an open right-hand endpoint, and so $b < d$ and $b' < d'$.
Moreover, by the definition of $U$, we cannot have $t_0 < b < t_0 + \delta$ or $t_0 < d < t_0 + \delta$.
Observe that the persistence modules $(M^\cR(X),M^\cR(X))$ are $(\tau(t)=t,\sigma(t)=t)$-interleaved and that $(M^\cR(Y),M^\cR(X))$ are $(\tau'(t)=t,\sigma'(t)=t+\delta)$-interleaved.
The result follows by applying Proposition~\ref{prop:stitching_lemma} and Proposition~\ref{prop:Matching}, with the additional observation that the identity map yielding $U_{t_0} = M^\cR(X)_{t_0}$ and the definition of $U$ forces that every point $[ t_0 , d, i ] \in \pd(U)$ is matched to some point $[t_0, d', i'] \in \pd({M^{\cal R}(X)})$ and vice versa.
For the reader's benefit, we indicate the forms of $\eta$, $\rho$, $\eta^\dagger$ and $\rho^\dagger$ in Figure~\ref{fig:etarho}.
\end{proof}

\subsubsection{Iterated subsampling of a large point cloud}
\label{subsec:pointcoud_subsample}

The goal of this section is to demonstrate that the techniques developed in Sections~\ref{subsec:merge_module} and \ref{subsec:pointcoud_subsample} can be used to obtain a multiscale approximation of the persistence diagram of a large point cloud $X$.
Our aim is to highlight the method as opposed to presenting an optimal result, and thus we begin with a sequence of $\delta$-approximations of $X$.

\begin{defn}
\label{defn:covertree}
Let $(X,d)$ be a finite metric space.
A sequence $\cY = \{Y_i \subseteq X\}_{i=0}^{m}$ is a \emph{$\Delta =\setof{\delta_i > 0}_{i=1}^m$ sampling} of $X$ if
\begin{enumerate}
\item[(i)] $Y_0 = X$, $Y_{i+1} \subset Y_{i}$, and $\delta_i < \delta_{i+1}$ for all $i$, and
\item[(ii)] for every $i>0$, $Y_{i}$ is a $\delta_i$ approximation of $Y_{i-1}$.
\end{enumerate}
\end{defn}

\begin{defn}
Let $ \cY = \{Y_i \subseteq X\}_{i=0}^{m}$ be a $\Delta = \setof{\delta_i > 0}_{i=1}^m$ sampling of a  finite metric space $(X,d)$.
An \emph{admissible stitching sequence} is a sequence $\cT = \setof{t_i \geq 0}_{i=1}^m$ satisfying
\[
t_{i+1} \geq t_{i}+\delta_{i},\quad i=1,\ldots, m-1.
\]
The associated \emph{stitched Vietoris-Rips persistence module} 
\[
U = U(M^\cR(X); \cY, \Delta, \cT)
\] 
is defined inductively as follows.

Using the fact that the pair $(M^\cR(X),M^\cR(X))$ is $(\tau_0,\sigma_0)$ interleaved, where $\tau_0(t) = t$ and $\sigma_0(t) = t$, Proposition~\ref{prop:subsample_interleaving}(i)  guarantees that $(M^\cR(Y_1),M^\cR(X))$ is $(\tau_1,\sigma_1)$ interleaved, where  $\tau_1(t)=t$ and $\sigma_1(t)=t+\delta_1$.
Set
\[
U_1:= U(M^\cR(X);M^\cR(X),t_1,M^\cR(Y_1)).
\]
By Proposition~\ref{prop:stitching_lemma}, $(U_1,M^\cR(X))$ is $(\eta_1,\rho_1)$ interleaved, where $\eta_1$ and $\rho_1$ are given as in Proposition~\ref{prop:stitching_lemma}(ii).

By Definition~\ref{defn:covertree}, $Y_i$ is a $\delta_i$-approximation of $X$, and hence by Proposition~\ref{prop:subsample_interleaving}(i), the persistence modules $(M^\cR(Y_{i}),M^\cR(X))$ are $(\tau_{i},\sigma_{i})$ interleaved where $\tau_{i}(t) = t$ and $\sigma_{i}(t) = t+\delta_i$.
Inductively, for $i=1,\ldots, m$, we use that the persistence modules $(U_{i-1},M^\cR(X))$ are $(\eta_{i-1},\rho_{i-1})$-interleaved and $(M^\cR(Y_{i}),M^\cR(X))$  are $(\tau_{i},\sigma_{i})$-interleaved  to define
\[
U_i := U(M^\cR(X);U_{i-1},t_{i},M^\cR(Y_{i}))
\]
and, again, use Proposition~\ref{prop:stitching_lemma}(ii) to define $(\eta_i,\rho_i)$ that gives the interleaving of the pair $(U_i,M^\cR(X))$.
\end{defn}

Observe that having constructed $U = U(M^\cR(X); \cY, \Delta, \cT),$ we have that  $(U,M^\cR(X))$ is $(\eta_m,\rho_m)$ interleaved.  
This implies that an analogue of Corollary~\ref{cor:subsample_interleaving} provides a quantitative comparison of the associated persistence diagrams. 
An explicit description of this corollary is probably of limited interest.
For many implementations, the computational bottleneck for obtaining a persistence diagram is the memory constraint associated with the Vietoris-Rips complex $\cR(X,t)$.
Thus, a desirable strategy is to: compute $M^\cR(X)$ over an interval $[0,t_0]$; downsample to $Y_1\subset X$; compute $U_1$ using $M^\cR(Y_1)$ over an interval $[t_0,t_1]$; downsample to $Y_2\subset Y_1$; and repeat the process.
An open question is how to optimize the choice of the locations $t_i$ of downsampling, and the $\delta_i \geq 0$ used to construct the downsampled sets $\setof{Y_i}_{i=1}^{m}$.

\subsection{A comparison of approximations of Vietoris-Rips and \u{C}ech filtrations}
\label{sec:comparison}
\begin{table}[t]
\label{table:interleavings}
\centering
 \renewcommand{\arraystretch}{1.5}
\begin{tabular}{| L | L | c | c |}\hline
\multicolumn{4}{| c |}{{\bf Point clouds in $\R^n$} } \\
\hline
{\it Approximate complex} & {\it Complex of Interest} & $\tau(t)$ & $\sigma(t)$ \\
\hline
Vietoris-Rips \cite[Proof of Vietoris-Rips Lemma.]{edelsbrunner} & \v{C}ech & $t\sqrt{\frac{2n}{n+1}}$ & $t$ \\
\hline
Net-tree \cite[Proof of Proposition 20]{botnan_spreeman} & \v{C}ech & $t$ & $(1+\varepsilon)^2t$ \\
\hline
Graph induced complex \cite[Proposition 2.8]{dey_fan_wang_2} & Vietoris-Rips & $t + 2\varepsilon$ & $t$ \\
\hline
Sparsified Vietoris-Rips \cite[Claim 6.1]{dey_fan_wang} & Vietoris-Rips & $t$ & $(1+\varepsilon)t$ \\
\hline
\multicolumn{4}{| c |}{{\bf Point clouds in an arbitrary metric space} } \\
\hline
{\it Approximate complex} & {\it Complex of Interest} & $\tau(t)$ & $\sigma(t)$ \\
\hline
Vietoris-Rips & \v{C}ech filtration & $2t$ & $t$ \\
\hline
Relaxed Vietoris-Rips \cite[Lemma 4]{sheehy} & Vietoris-Rips & $t$ & $\left(\frac{1}{1-2\varepsilon}\right)t$ \\
\hline
Sparse weighted Rips \cite[Lemma 6.13]{buchet} & Vietoris-Rips & $t$ & $\left( \frac{1 + \sqrt{1+\delta^2}\varepsilon}{1 - \varepsilon}\right)t$ \\
\hline
\end{tabular}
\caption{\emph{A table of approximations for Vietoris-Rips and \u{C}ech filtrations for point clouds.}
The first column gives the approximation and a reference to the construction of the approximation (i.e. the explicit construction of the interleavings), and the second column gives the complex that is being approximated.
The third and fourth columns list the translation maps for $(\tau, \sigma)$-interleavings of the associated persistence modules induced by taking homology of the associated filtrations.
The values $\delta, \varepsilon \geq 0$ are parameters specified by the approximations where applicable.}
\label{table:approx}
\end{table}

In applications, a persistence module is associated to  a finite metric space $(X,d)$ via the construction of a simplicial complex. 
There is typically a natural choice of complex for the problem of interest (e.g.  \u{C}ech complex).  
However, the Vietoris-Rips complex is usually more manageable than the  \u{C}ech complex. 
Table~\ref{table:approx} provides a list of examples of pairs of filtrations and their approximations that have appeared in the literature. 
Proposition~\ref{prop:Matching} provides a general quantitative comparison of  persistence diagrams given an interleaving between the associated persistence modules. 

Table~\ref{table:approx} explicitly  defines interleavings and the interested reader can derive the bounds for the matching of persistence diagrams using Corollary~\ref{cor:genalgstab_invertible_pi} since all of the maps $\tau$ and $\sigma$ in the table are bijections.
Note that Proposition~\ref{prop:interleaving_composition} enables one to keep track of errors even when multiple approximation steps have been used. For example, say that one desires to make a statement about the persistence diagram corresponding to the \u{C}ech filtration of a finite point cloud in $\R^n$ via the persistence diagram corresponding to a filtration of the Sparsified Vietoris-Rips complex from \cite{dey_fan_wang} with parameter $\varepsilon$. Then the $(\eta, \rho)$-interleaving between the persistence module induced by the \u{C}ech filtration and the persistence module induced by the Sparsified Vietoris-Rips complex filtration is given by $\eta(t) = t\sqrt{\frac{2n}{n+1}}$ and $\rho(t) = (1 + \varepsilon)t$, where an intermediate approximation uses the Veitoris-Rips complex filtration (e.g. the translation pairs that one plugs into Proposition~\ref{prop:interleaving_composition} come from the first and last rows of the first section of the table).

\section{Acknowledgements}
\label{sec:acknowledgements}
R. L. would like to thank Chuck Weibel, Michael Lesnick, and Ulrich Bauer for the many insightful discussions that led to the results presented in this paper.

M. K. was supported by ERC project GUDHI (Geometric Understanding in Higher Dimensions).
R. L. was supported by DARPA contracts HR0011-17-1-0004 and HR0011-16-2-0033.
S. H. and K. M. were partially supported by grants NSF-DMS-1125174, 1248071, 1521771,
NIH 1R01GM126555-01 and DARPA contracts HR0011-16-2-0033, FA8750-17-C-0054.


\section*{References}
\label{sec:References}
\bibliographystyle{plain}

\end{document}